\documentclass[11pt]{amsart}
\usepackage{amssymb, amsmath, amscd, eucal, rotating}
\usepackage{pstricks,pstricks-add,pst-node,pst-plot,pst-grad,pst-math,pstcol,pst-3dplot,pst-solides3d}
\usepackage{graphicx}
\usepackage{times,epsf,epsfig,color}

\input xy
 \xyoption{all}

\numberwithin{equation}{section}

\theoremstyle{plain}

\newtheorem{proposition}{Proposition}[section]
\newtheorem{theorem}[proposition]{Theorem}		
\newtheorem*{theorem*}{Theorem}		
\newtheorem{corollary}[proposition]{Corollary}
\newtheorem{lemma}[proposition]{Lemma}

\theoremstyle{definition}

\newtheorem{definition}[proposition]{Definition}
\newtheorem{remark}[proposition]{Remark}
\newtheorem{example}[proposition]{Example}

\newcommand{\C}{\mathbb C}
\newcommand{\R}{\mathbb R}

\newcommand{\Hom}{\mathop{\rm Hom}\nolimits}

\newcommand{\Ad}{\mathop{\rm Ad}\nolimits}

\newcommand{\GL}{\mathsf{GL}}

\newcommand{\U}{\mathsf{U}}

\DeclareMathOperator{\Rep}{Rep}

\DeclareMathOperator{\id}{id}
\DeclareMathOperator{\im}{im}

\DeclareMathOperator{\grad}{grad}

\DeclareMathOperator{\Crit}{Crit}

\begin{document}


\title{Equivariant Morse theory for the norm-square of a moment map on a variety}

\author{Graeme Wilkin}
\address{Department of Mathematics,
National University of Singapore, 
Singapore 119076}
\email{graeme@nus.edu.sg}

\date{\today}

\thanks{This research was partially supported by grant number R-146-000-200-112 from the National University of Singapore. The author also acknowledges support from NSF grants DMS 1107452, 1107263, 1107367 ``RNMS GEometric structures And Representation varieties'' (the GEAR Network).}


\begin{abstract}
We show that the main theorem of Morse theory holds for a large class of functions on singular spaces. The function must satisfy certain conditions extending the usual requirements on a manifold that Condition C holds and the gradient flow around the critical sets is well-behaved, and the singular space must satisfy a local deformation retract condition. We then show that these conditions are satisfied when the function is the norm-square of a moment map on an affine variety, and that the homotopy equivalence from this theorem is equivariant with respect to the associated Hamiltonian group action. An important special case of these results is that the main theorem of Morse theory holds for the norm square of a moment map on the space of representations of a finite quiver with relations.
\end{abstract}


\maketitle


\thispagestyle{empty}

\baselineskip=16pt



\section{Introduction}

Morse theory relates information about the topology of a manifold $M$ to information about the critical points of a smooth Morse function $f : M \rightarrow \R$. The ``Main theorem of Morse theory" makes this precise by describing the change in the homotopy type of the level sets of $f$ in terms of analytic data (the Morse index) around the critical points. When $f$ is a proper Morse function, the statement is as follows (see for example \cite{Bott88}).
\begin{theorem*}[Main theorem of Morse theory]\label{thm:Milnor-main-theorem}
Let $f : M \rightarrow \R$ be a $C^\infty$ proper function with nondegenerate critical points, let $e_\lambda$ denote a cell of dimension $\lambda$ and given any $a \in \R$, let $M_a = f^{-1}(-\infty, a]$. 
\begin{enumerate}
\item If there are no critical values of $f$ in $[a,b]$ then $M_b \simeq M_a$.

\item If $a < c < b$ and there is one critical value $c$ in $[a,b]$ of index $\lambda$ then $M_b \simeq M_a \cup e_\lambda$.
\end{enumerate}
\end{theorem*}


Various generalisations of this idea have been used to (a) deduce information about the critical points of a given function from information about the topology of a manifold (for example \cite{Morse25}, \cite{Morse28}, \cite{Milnor63}, \cite{Palais63}, \cite{Smale64}, \cite{GromollMeyer69}, \cite{Conley78}, \cite{ConleyZehnder84}, \cite{Floer88}) and (b) deduce information about the topology of a manifold from information about the critical points of a function (for example \cite{Bott56}, \cite{BottSamelson58}, \cite{Smale61}, \cite{AtiyahBott83}, \cite{Kirwan84}).

Each of these results depends on proving an analog of the main theorem of Morse theory for a given class of functions. The generalisation most relevant to this paper is that of Kirwan in \cite[Sec. 10]{Kirwan84} who proved that an analogous statement holds for minimally degenerate functions on a smooth manifold (``Morse-Kirwan functions''). Kirwan also showed that the norm-square of a moment map on a smooth symplectic manifold is always minimally degenerate and used this to derive very general results about the topology of symplectic quotients.

This idea originated in the work of Atiyah and Bott \cite{AtiyahBott83}, who used the Morse theory of the Yang-Mills functional in order to inductively compute cohomological invariants of the moduli space of semistable holomorphic bundles on a compact Riemann surface. Their approach was algebraic in the sense that they used the Harder-Narasimhan stratification instead of the gradient flow stratification by the Yang-Mills functional, and the analytic details of the Morse theory were later filled in by Daskalopoulos \cite{Daskal92} and Rade \cite{Rade92}. In \cite{DWWW11}, \cite{WentworthWilkin12} and \cite{WentworthWilkin13} we continued this program for certain spaces of coupled equations in gauge theory in order to derive new results about the topology of moduli spaces of Higgs bundles and stable pairs in low rank. In these examples the analog of the Morse function is the Yang-Mills-Higgs functional, which can be interpreted as the norm-square of a moment map (cf. \cite{AtiyahBott83}, \cite{Hitchin87}, \cite{Bradlow91}). The new feature  is that the underlying space is singular and since the existing theory from \cite{AtiyahBott83} and \cite{Kirwan84} requires the underlying space to be a manifold, then a new proof of the main theorem of Morse theory is needed, which we carried out using a method specific to these particular examples in low rank. Many other interesting examples of moduli spaces (for example the quiver varieties of \cite{KronheimerNakajima90}, \cite{Nakajima94} and \cite{Nakajima12}) can also be defined as the minimum of the norm-square of a moment map on a singular space, and therefore it is of interest to develop a general approach for these examples.

The first result of this paper is Theorem \ref{thm:main-theorem-Morse-theory}, which shows that the main theorem of Morse theory holds for a large class of functions on singular spaces (functions $f : Z \rightarrow \R$ satisfying Conditions \eqref{cond:isolated-crit-values}--\eqref{cond:def-retract} below). In particular, this class of functions includes the case where $f$ is the norm-square of a moment map on an affine variety (not necessarily smooth), which extends Kirwan's results \cite[Sec. 10]{Kirwan84} from smooth varieties to varieties with singularities. An important class of examples is given by representations of a finite quiver with relations, and this paper together with \cite{wilkin-quiver-flow-hecke} completes the study of the local analysis around the critical sets with a view to using the ideas of Atiyah \& Bott and Kirwan to compute topological invariants of moduli spaces of quivers with relations in analogy with the approach of \cite{DWWW11}, \cite{WentworthWilkin12} and \cite{WentworthWilkin13}.

The results are stated as follows. Let $M$ be a Riemannian manifold and let $f : M \rightarrow \mathbb{R}$ be a smooth function. The time $t$ negative gradient flow of $f$ with initial condition $x \in M$ is denoted by $\phi(x, t)$ satisfying
\begin{equation*}
\frac{\partial}{\partial t} \phi(x, t) = - \grad f(\phi(x,t)), \quad \phi(x,0) = x .
\end{equation*}
Let $Z \subset M$ be any closed subset preserved by the gradient flow of $f$, i.e. if $x \in Z$ then $\phi(x, t) \in Z$ for all $t \in \mathbb{R}$ such that $\phi(x, t)$ is defined. It is not necessary to assume that $\phi(x, t)$ exists for all $t \in \R$, however when the initial condition is in $Z$ then we do assume local existence, continuous dependence on initial conditions and Condition \eqref{cond:flow-alternative} below. Since $Z$ is closed in $M$ then if $x \in Z$ and $\lim_{t \rightarrow \infty} \phi(x, t)$ or $\lim_{t \rightarrow - \infty} \phi(x, t)$ exists in $M$ then this limit is contained in $Z$.

Define a critical point of $f : Z \rightarrow \mathbb{R}$ to be a stationary point of this flow (i.e. a critical point in $Z$ is a critical point of $f : M \rightarrow \mathbb{R}$ that is also contained in $Z$). Let $\Crit_Z(f)$ denote the set of critical points for $f : Z \rightarrow \R$. Given a critical value $c \in \R$ and the associated critical set $C = \Crit_Z(f) \cap f^{-1}(c)$, define $W_C^+$ and $W_C^-$ to be the stable and unstable sets of $C$ with respect to the flow
\begin{align*}
W_C^+ & := \left\{ x \in Z \mid \lim_{t \rightarrow \infty} \phi(x, t) \in C \right\} \\
W_C^- & := \left\{ x \in Z \mid \lim_{t \rightarrow - \infty} \phi(x, t) \in C \right\} .
\end{align*}
Let $W_x^+$ and $W_x^-$ denote the analogous stable/unstable sets with respect to a specific critical point $x \in C$. Suppose also that $f : Z \rightarrow \R$ and the flow $\phi$ satisfy the following conditions.
\begin{enumerate}

\item \label{cond:isolated-crit-values} The critical values of $f$ are isolated.

\item \label{cond:flow-alternative} For any $a < b$ such that $f^{-1}(a) \cap \Crit_Z(f) = \emptyset = f^{-1}(b) \cap \Crit_Z(f)$ and any $x \in f^{-1}((a,b))$ either there exists $T_+ > 0$ such that $\phi(x, T_+) = a$ or $\lim_{t \rightarrow \infty} \phi(x, t)$ exists in $f^{-1}((a, b))$. Similarly, either there exists $T_- < 0$ such that $\phi(x, T_-) = b$ or $\lim_{t \rightarrow -\infty} \phi(x, t)$ exists in $f^{-1}((a,b))$.

\item \label{cond:analytic} $M$ is real analytic and $f : M \rightarrow \mathbb{R}$ is real analytic.

\item \label{cond:hyperbolicity} For each non-minimal critical point $x$, let $C = \Crit_Z(f) \cap f^{-1}(f(x))$. For each $a < f(x)$ such that there are no critical values in $[a, f(x))$ and for each neighbourhood $U \subset f^{-1}(a)$ of $W_x^- \cap f^{-1}(a)$ there exists a neighborhood $V$ of $x$ such that for each $y \in V \setminus W_C^+$ there exists $T > 0$ such that $\phi(y, T) \in U$.

\item \label{cond:def-retract} For each critical value $c$ let $C = \Crit_Z(f) \cap f^{-1}(c)$. Then there exists $\varepsilon > 0$ such that $W_C^- \cap f^{-1}(c-\varepsilon)$ has an open neighborhood $E \subset f^{-1}(c-\varepsilon)$ and a strong deformation retract $r : E \times [0,1] \rightarrow E$ of $E$ onto $W_C^- \cap f^{-1}(c-\varepsilon)$ such that (using $E_s$ to denote the image $r(E, s)$ for each $s \in [0,1]$) we have
\begin{enumerate}

\item $E_s$ is open in $f^{-1}(c-\varepsilon)$ for all $s \in [0,1)$,

\item $\overline{E}_t \subset E_s$ for all $t > s$,

\item $E_s = \bigcup_{t > s} E_t$ and $\overline{E}_s = \bigcap_{t < s} E_t$ for all $s \in (0,1)$.

\end{enumerate}

\end{enumerate}

 Denote $Z_a := \{ x \in Z \mid f(x) \leq a \}$. The following theorem is the analog of the main theorem of Morse theory for $f : Z \rightarrow \R$.

\begin{theorem}\label{thm:main-theorem-Morse-theory}
Let $M$ be a Riemannian manifold and let $f : M \rightarrow \mathbb{R}$ be a smooth function. Let $Z \subset M$ be any closed subset preserved by the gradient flow of $f$ and suppose that the restriction $f : Z \rightarrow \mathbb{R}$ satisfies the conditions \eqref{cond:isolated-crit-values}--\eqref{cond:def-retract}. 
\begin{enumerate}

\item[(a)] \label{item:main-theorem-no-critical} If there are no critical values in $[a,b]$ then $Z_a \simeq Z_b$.

\item[(b)] \label{item:main-theorem-critical} If $f^{-1}(a) \cap \Crit_Z(f) = \emptyset = f^{-1}(b) \cap \Crit_Z(f)$ and there is only one critical value $c \in (a,b)$ with associated critical set $C = \Crit_Z(f) \cap f^{-1}(c)$ then $Z_a \cup W_C^- \simeq Z_b$.

\end{enumerate}

Moreover, if a group $K$ acts on $Z$ such that (a) $f$ is $K$-invariant and (b) the deformation retract of condition \eqref{cond:def-retract} is $K$-equivariant, then these homotopy equivalences are $K$-equivariant.
\end{theorem}

Instead of directly verifying Condition \eqref{cond:def-retract} for each example, Theorem 1.1 of \cite{pflaumwilkin} shows that it is sufficient to verify the following simpler condition, which can be done for a large class of examples.
\begin{enumerate}
\item[(5$'$)] For each critical value $c$ there exists $\varepsilon > 0$ such that there are no critical values of $f$ in $[c-\varepsilon, c)$, and there is a stratification of $f^{-1}(c-\varepsilon)$ satisfying Whitney's Condition B such that $W_C^- \cap f^{-1}(c-\varepsilon)$ is a union of strata.
\end{enumerate}
Moreover, Theorem 1.1 of \cite{pflaumwilkin} shows that if a compact Lie group $K$ acts on $f^{-1}(c-\varepsilon)$ and preserves $W_C^- \cap f^{-1}(c-\varepsilon)$, then the deformation retract of Condition (5) can be made to be $K$-equivariant.

In Section \ref{sec:condition5-moment-maps} we prove that $(5')$ is satisfied in the case where $f : Z \rightarrow \R$ is the norm-square of a moment map on an affine variety. As a consequence of this, we can prove that the main theorem of Morse theory holds in this setting. 
\begin{theorem}\label{thm:main-theorem-affine-variety}
Let $G$ be a connected complex reductive group and let $V$ be a linear representation of $G$. Suppose that the action of the maximal compact subgroup $K \subset G$ is Hamiltonian with respect to the standard symplectic structure on $V$ and let $\mu : V \rightarrow \mathfrak{k}^*$ be a moment map for this action. Let $Z \subset V$ be a closed affine subvariety preserved by $G$. Then $f = \| \mu \|^2 : Z \rightarrow \mathbb{R}$ satisfies conditions \eqref{cond:isolated-crit-values}--\eqref{cond:def-retract} and therefore
\begin{enumerate}

\item[(a)] \label{item:moment-map-no-critical} If there are no critical values in $[a,b]$, then $Z_a$ is $K$-equivariantly homotopic to $Z_b$.

\item[(b)] \label{item:moment-map-critical} If $f^{-1}(a) \cap \Crit_Z(f) = \emptyset = f^{-1}(b) \cap \Crit_Z(f)$ and there is only one critical value $c \in (a,b)$ with associated critical set $C = \Crit_Z(f) \cap f^{-1}(c)$, then $Z_a \cup W_C^-$ is $K$-equivariantly homotopic to $Z_b$.

\end{enumerate} 

\end{theorem}

\begin{remark}

A result of Kempf \cite[Lemma 1.1]{Kempf78} shows that for an affine variety $Z$ with the action of a connected reductive algebraic group $G$, there is a representation $V$ of $G$ and a $G$-equivariant isomorphism from $Z$ to a closed affine subvariety of $V$, and so the above theorem applies to any affine $G$-variety.

\end{remark}

{\bf Conditions \eqref{cond:isolated-crit-values} and \eqref{cond:flow-alternative} in the context of Morse theory on smooth spaces.} Combining the results of \cite[Sec. 10]{Palais63} and \cite{Simon83} shows that Conditions \eqref{cond:isolated-crit-values} and \eqref{cond:flow-alternative} are automatically satisfied by an analytic function satisfying the Palais-Smale Condition C. The advantage of Condition C is that it can be verified directly from the function $f$ (i.e. one does not need solve the gradient flow equations to verify Condition C), however many of the examples that we would like to study do not satisfy Condition C because the critical sets are non-compact. For example, in the setting of Theorem \ref{thm:main-theorem-affine-variety} the norm-square of a moment map on the space of representations of a quiver with an oriented cycle does not satisfy Condition C (not even equivariantly), however the norm square of a moment map on an affine variety does satisfy Conditions \eqref{cond:isolated-crit-values} and \eqref{cond:flow-alternative} above, since the gradient flow satisfies a compactness condition due to Sjamaar \cite{Sjamaar98} (see Proposition \ref{prop:conditions1-4}) which allows us to prove that Conditions \eqref{cond:isolated-crit-values} and \eqref{cond:flow-alternative} hold in this setting.

{\bf Conditions \eqref{cond:analytic} and \eqref{cond:hyperbolicity} in the context of Morse theory on smooth spaces.} Conditions \eqref{cond:analytic} and \eqref{cond:hyperbolicity} are chosen to ensure the continuity of the homotopy equivalences near the critical set (see Propositions \ref{prop:retract-to-critical-level} and \ref{prop:retract-to-cell-attachment}). These conditions impose extra structure on the maps between level sets of $f$ defined by the gradient flow (see Lemma \ref{lem:level-set-map-closed}), which allows for the gradient flow $\phi(x,t)$ to translate the deformation retract of Condition \eqref{cond:def-retract} from a neighbourhood of the unstable manifold to a neighbourhood of the critical set.

In the context of moment maps on affine varieties, it is natural to impose Condition \eqref{cond:analytic}, since the norm-square of a moment map on an affine variety is an analytic function. It is also natural to impose Condition \eqref{cond:hyperbolicity}, since this condition is satisfied by any function whose gradient flow is hyperbolic around the critical set. In particular this is true for a Morse function or a Morse-Bott function, since the flow is hyperbolic with respect to the coordinates given by the Morse Lemma (see for example \cite{HirschPughShub77}, \cite{Jost05}). Since Condition \eqref{cond:hyperbolicity} only involves the unstable set $W_C^-$ and not the stable set $W_C^+$ then it is also satisfied by a Morse-Kirwan function (see \cite[Sec. 10]{Kirwan84} and Proposition \ref{prop:conditions1-4} in this paper for more details). 

Therefore Conditions \eqref{cond:isolated-crit-values}--\eqref{cond:hyperbolicity} are natural extensions of the usual conditions needed to prove the main theorem of Morse theory for functions on a manifold (cf. \cite{Milnor63}, \cite{Palais63} or \cite{Kirwan84}).  Lemma \ref{lem:restrict-to-closed-subset} shows that when $Z$ is closed in $M$ and preserved by the gradient flow $\phi$, then Conditions \eqref{cond:isolated-crit-values}--\eqref{cond:hyperbolicity} can be checked by studying the properties of $f$ on the ambient smooth manifold $M$. Therefore, in order to show that these conditions are satisfied for the norm-square of a moment map on an affine variety, it is sufficient to check these conditions for the moment map associated to a linear Hamiltonian action on a vector space, which is done in Proposition \ref{prop:conditions1-4}.

The structure of the singular set $Z$ enters the picture via Condition \eqref{cond:def-retract}. The deformation retract studied here is a special case of a Neighbourhood Deformation Retract (NDR) (see for example \cite{May99}), however we also require the extra conditions (a), (b) and (c) on the deformation retract in order to guarantee that the function $\sigma$ from Section \ref{subsec:sigma} is continuous (this is explained in Lemma \ref{lem:sigma-continuous}), which is needed to show that the deformation retract of \eqref{eqn:definition-R} is continuous. Proposition \ref{prop:condition5} shows that Condition \eqref{cond:def-retract} is satisfied when $f$ is the norm square of a moment map on an analytic variety and that the deformation retract can be chosen to be equivariant with respect to the associated Hamiltonian group action. Therefore Condition \eqref{cond:def-retract} is valid for a large class of interesting examples. 

{\bf Connection with other examples of Morse theory on singular spaces.} The stratified Morse theory of Goresky and MacPherson \cite{GoreskyMacPhersonSMT} is also valid for a large class of functions on singular spaces which includes affine varieties. This theory uses a Whitney stratification of the singular space $Z$ which is compatible with the function $f : Z \rightarrow \R$. It is important to note that in general the norm-square of a moment map on a variety does not fulfil the conditions of Goresky and MacPherson in \cite{GoreskyMacPhersonSMT}. It may be possible to perturb the original function to obtain a new function satisfying Goresky and MacPherson's conditions, however in doing this we lose the equivariance of the moment map and also lose the possibility of inductively computing the cohomology of the critical sets in analogy with the computations of Atiyah \& Bott \cite{AtiyahBott83} and Kirwan \cite{Kirwan84} when the space is smooth.

The essential difference between the stratified Morse theory of \cite{GoreskyMacPhersonSMT} and the results of this paper is that Goresky and MacPherson use the local structure of the Whitney stratification to prove the main theorem of Morse theory, while for moment maps on affine varieties we already have a gradient flow which is well-behaved near the critical sets, and so we use the properties of this flow to prove Theorem \ref{thm:main-theorem-Morse-theory} instead of using the properties of the Whitney stratification.

The Conley index theory \cite{Conley78} is another theory valid for singular spaces, however the proof of the homotopy invariance of the Conley index requires the critical sets to have compact neighbourhoods, which is not generally true for the norm-square of a moment map on an affine variety. In particular, the proof of 4.2(D) on p50 of \cite{Conley78} requires compactness (see also \cite[Lemma 4.7]{Salamon85}). Nicolaescu in \cite[Thm. 9.10]{Nicolaescu10} also uses the homotopy invariance of the Conley index to prove an analog of the main theorem of Morse theory for tame flows with Morse-like critical points, however this proof also requires compactness. It may be possible to recover the Conley theory for moment maps on affine varieties by carefully analysing the behaviour of the gradient flow of $f$ near the ends of the critical sets, but we avoid this approach here as this would require specific knowledge of $f$, and the theory of this paper only requires checking Conditions \eqref{cond:isolated-crit-values}--\eqref{cond:def-retract} which are already valid for a large class of examples.

{\bf Organisation of the paper.} The results of Section \ref{sec:preliminary} show that it is possible to deformation retract to a ``good'' neighbourhood of each critical set, and the results of Section \ref{sec:main-theorem-proof} show that it is possible to deformation retract from this neighbourhood to the unstable manifold, which completes the proof of Theorem \ref{thm:main-theorem-Morse-theory}. In Section \ref{sec:moment-map} we show that Conditions \eqref{cond:isolated-crit-values}--\eqref{cond:hyperbolicity} are satisfied when $f : Z \rightarrow \R$ is the norm-square of a moment map on a variety and in Section \ref{sec:compactness-flow-lines} we use Conditions \eqref{cond:isolated-crit-values}--\eqref{cond:hyperbolicity} to prove a compactness theorem for sequences of flow lines. The results of Section \ref{sec:condition5-moment-maps} complete the proof of Theorem \ref{thm:main-theorem-affine-variety} by showing that Condition \eqref{cond:def-retract} holds for moment maps on varieties.

{\bf Acknowledgements.} I would like to thank Dinh Tien Cuong and Markus Pflaum for sharing their knowledge of singular spaces, Carlos Florentino for pointing out the reference \cite{Kempf78}, and George Daskalopoulos and Richard Wentworth for discussions about our joint work \cite{DWWW11}, \cite{WentworthWilkin12} and \cite{WentworthWilkin13} which motivated the current project. 

\section{Preliminary results}\label{sec:preliminary}

This section contains the preliminary results and definitions needed to complete the proof of Theorem \ref{thm:main-theorem-Morse-theory}. The first two steps of the proof of Theorem \ref{thm:main-theorem-Morse-theory} are contained in Proposition \ref{prop:retract-to-critical-level} and Proposition \ref{prop:retract-to-Y}. Throughout this section and the next we will refer to Conditions \eqref{cond:isolated-crit-values}--\eqref{cond:def-retract} from the introduction.

\subsection{The deformation retraction defined by the gradient flow}

\begin{lemma}\label{lem:tau-continuous}
Suppose that $f : Z \rightarrow \R$ satisfies Condition \eqref{cond:flow-alternative}, and suppose also that there is at most one critical value $c$ in the interval $[a,b]$. Let $C = \Crit_Z(f) \cap f^{-1}(c)$. Then for each $\ell \in [a,b]$ there is a continuous function $\tau_\ell : f^{-1}([a,b]) \setminus (W_C^+ \cup W_C^-) \rightarrow \mathbb{R}$ such that $f(\phi(x, \tau_\ell(x))) = \ell$. 

If $\ell \in [a,c)$ then $\tau_\ell$ can be extended to a continuous function $\tau_\ell : f^{-1}([a,b]) \setminus W_C^+ \rightarrow \R$.

If $\ell \in (c,b]$ then $\tau_\ell$ can be extended to a continuous function $\tau_\ell : f^{-1}([a,b]) \setminus W_C^- \rightarrow \R$.

Moreover, $\tau_\ell(x)$ also depends continuously on $\ell$.
\end{lemma}

\begin{proof}
For each $x \in  f^{-1}([a,b]) \setminus (W_C^+ \cup W_C^-)$, the function $\tau_\ell(x)$ exists by Condition \eqref{cond:flow-alternative}. Since $f(\phi(x, t)) \in \R$ depends continuously on $(x, t)$ and is $C^1$ with respect to $t$ with $\frac{\partial}{\partial t} f(\phi(x, t)) < 0$, then $\tau_\ell(x)$ is uniquely defined and continuous with respect to $x$.

If $\ell \in [a,c)$ and $x \in f^{-1}([a,b]) \setminus W_C^+$ then $\tau_\ell(x)$ exists by Condition \eqref{cond:flow-alternative} and the same proof shows that $\tau_\ell$ is well-defined and continuous with respect to $x$ on  $f^{-1}([a,b]) \setminus W_C^+$. Similarly, if $\ell \in (c,b]$ then $\tau_\ell$ is well-defined and continuous with respect to $x$ on $f^{-1}([a,b]) \setminus W_C^- \rightarrow \R$.

To show continuous dependence on $\ell$, note that 
\begin{equation}\label{eqn:tau_ell-in-terms-of-ell}
\ell = f(\phi(x,\tau_\ell(x))) = f(x) - \int_0^{\tau_\ell(x)} \| \grad f(\phi(x,t)) \|^2 \, dt
\end{equation}
Since $\| \grad f(\phi(x,t)) \| > 0$ depends continuously on $(x,t)$ then for $t$ in any closed bounded interval we have that $\| \grad f(\phi(x, t)) \|$ is bounded below by a positive constant (here we use the continuity of $\grad f$ along a compact subset of a flow line in place of Condition C as in \cite{Palais63}). Since we assumed that $x \notin W_C^+ \cup W_C^-$, then  $\| \grad f(\phi(x, t)) \|$ is bounded below by a positive constant along the entire flow line from $x$ to $\phi(x, \tau_\ell(x))$ and therefore \eqref{eqn:tau_ell-in-terms-of-ell} implies that $\tau_\ell(x)$ depends continuously on $\ell$. Similarly, if $\ell \in [a,c)$ and $x \notin W_C^+$, or if $\ell \in (c,b]$ and $x \notin W_C^-$, then the same argument shows that $\tau_\ell(x)$ depends continuously on $\ell$.
\end{proof}

The following lemma is well-known (see for example \cite[Sec. 10]{Palais63} or \cite[Thm. 2.3]{Conley78}).
\begin{lemma}\label{lem:retract-between-critical-levels}
If $f : Z \rightarrow \R$ satisfies Conditions \eqref{cond:isolated-crit-values} and \eqref{cond:flow-alternative}, and there are no critical values in the interval $[a,b]$, then there is a deformation retract $Z_b \simeq Z_a$ . 
\end{lemma}

\begin{remark}
Note that Conditions \eqref{cond:isolated-crit-values} and \eqref{cond:flow-alternative} are a consequence of the Palais-Smale Condition C used in \cite[Sec. 10]{Palais63} and the proof of Lemma \ref{lem:retract-between-critical-levels} given in \cite{Palais63} only uses Conditions \eqref{cond:isolated-crit-values} and \eqref{cond:flow-alternative}. We avoid assuming Condition C here since we would also like to consider the case where $f$ is the norm-square of a moment map on a variety, in which case Condition C fails in general since $f$ may have non-compact critical sets, but Conditions \eqref{cond:isolated-crit-values} and \eqref{cond:flow-alternative} still hold.
\end{remark}

The next result shows that if we also assume Condition \eqref{cond:analytic} ($f$ is analytic), then an analogous statement is true for the situation when there is a critical value at one end of the interval $[c,b]$.

\begin{proposition}\label{prop:retract-to-critical-level}
Suppose that $f : Z \rightarrow \mathbb{R}$ satisfies conditions \eqref{cond:isolated-crit-values}, \eqref{cond:flow-alternative} and \eqref{cond:analytic}. Let $C$ be a critical set of $f : Z \rightarrow \mathbb{R}$ with critical value $c$ and suppose that there are no critical values in the interval $(c, b]$. Then $Z_b \simeq Z_c$.
\end{proposition}

\begin{proof}
Lemma \ref{lem:tau-continuous} shows that there exists a continuous function $\tau_c : f^{-1}([c,b]) \setminus W_C^+ \rightarrow \R$ such that $f(x, \tau_c(x)) = c$. This immediately gives us a deformation retract from $f^{-1}([c,b]) \setminus W_C^+$ to $f^{-1}(c) \setminus C$, and the goal of the proof is to show that this can be extended to a deformation retract from $f^{-1}([c,b])$ to $f^{-1}(c)$.

Given $\ell_1, \ell_2 \in (c,b]$, Lemma \ref{lem:tau-continuous} implies that the gradient flow defines a continuous map of level sets $\varphi_{\ell_1, \ell_2} : f^{-1}(\ell_1) \rightarrow f^{-1}(\ell_2)$ by $\varphi_{\ell_1,\ell_2}(x) = \phi(x, \tau_{\ell_2}(x))$. 

If $x \in W_C^+$ then Condition \eqref{cond:flow-alternative} guarantees that $\lim_{t \rightarrow \infty} \phi(x, t) = z$ for some $z \in C$. Therefore, for any $\ell \in (c,b]$ there is a well-defined map of level sets $\varphi_{\ell, c} : f^{-1}(\ell) \rightarrow f^{-1}(c)$ and we aim to show that this is continuous. If $x \in f^{-1}(\ell) \setminus (f^{-1}(\ell) \cap W_C^+)$ then $\tau_c(x)$ is finite and so the continuity of $\varphi_{\ell,c}(x) = \phi(x, \tau_c(x))$ follows from the continuous dependence of $\tau_c$ on $x$ proved in Lemma \ref{lem:tau-continuous}. 

If $x \in f^{-1}(\ell) \cap W_C^+$ then the proof of continuity of $\varphi_{\ell,c}$ uses the Lojasiewicz gradient inequality method of \cite{Simon83} as follows. Let $z = \varphi_{\ell,c}(x)$. For every neighbourhood $V$ of $z$ in $f^{-1}(c)$ there exists a neighbourhood $U$ of $z$ in $f^{-1}([c, \infty))$ such that if $y \in U$ then either $\phi(y,t)$ converges to a critical point in $V \cap C$, or there exists a finite $T$ such that $\phi(y, T) \in V$. Continuity of the finite-time flow guarantees an open neighbourhood $U'$ of $x$ in $f^{-1}(\ell)$ such that for each $y \in U'$ there exists $T'$ such that $\phi(y, T') \in U$ and therefore $\varphi_{\ell, c}(y) \in V$. Therefore, given any open set $V \subset f^{-1}(c)$ containing $\varphi_{\ell,c}(x)$, there exists an open neighbourhood $U'$ of $x$ in $f^{-1}(\ell)$ such that $\varphi_{\ell, c}(U') \subset V$, and so $\varphi_{\ell, c}$ is continuous.

Therefore we can define a continuous deformation retract $\rho : Z_b \times [0,1] \rightarrow Z_c$ by
\begin{equation*}
\rho(x,s) = \begin{cases} x & \text{if} \, \, f(x) \leq c \\ \varphi_{f(x), (f(x) - s(f(x)-c))}(x) & \text{if} \, \, c < f(x) \leq b \end{cases} \qedhere
\end{equation*}
\end{proof}

\subsection{The deformation retract to a neighbourhood of the critical set}\label{subsec:sigma}

In this section we prove Proposition \ref{prop:retract-to-Y}, which shows that $Z_c$ deformation retracts to the union of $Z_{c-\varepsilon}$ with the set $Y$ defined below.

Now assume that conditions \eqref{cond:isolated-crit-values}, \eqref{cond:flow-alternative}, \eqref{cond:hyperbolicity} and \eqref{cond:def-retract} hold. Let $c$ be a critical value with corresponding critical set $C = \Crit_Z(f) \cap f^{-1}(c)$ and let $\varepsilon$, $E$ and $r : E \times [0,1] \rightarrow E$ be as in condition \eqref{cond:def-retract}. Define $E_s = r(E, s)$ for all $s \in [0,1]$. Define $\sigma : f^{-1}(c-\varepsilon) \rightarrow [0,1]$ by
\begin{equation*}
\sigma(x) := \begin{cases} 0 & x \notin E \\ \sup \{ s \in [0,1] \mid x \in E_s \} & x \in E \end{cases}
\end{equation*}

\begin{lemma}\label{lem:sigma-continuous}
If $x \in E$ and $s \in (0,1)$, then
\begin{enumerate}
\item $\sigma(x) = s$ if and only if $x \in \overline{E}_s \setminus E_s$.

\item $\sigma(x) < s$ if and only if $x \notin \overline{E}_s$.

\item $\sigma(x) > s$ if and only if $x \in E_s$.
\end{enumerate}
\end{lemma}

\begin{proof}

First note that since $\overline{E}_t \subset E_s$ for all $t > s$ then the set $\{ s \in [0,1] \mid x \in E_s \}$ is a connected interval for all $x \in E$.
\begin{enumerate}

\item If $\sigma(x) = s$ then $x \in E_t$ for all $t < s$. Therefore $x \in \bigcap_{t < s} E_t = \overline{E}_s$. If $x \in E_s = \bigcup_{t > s} E_t$ then there exists $t > s$ such that $x \in E_t$, contradicting $\sigma(x) = s$. Therefore $x \in \overline{E}_s \setminus E_s$.

Conversely, if $x \in \overline{E}_s = \bigcap_{t < s} E_t$ then $\sigma(x) \geq s$. If $x \notin E_s$ then $\{ s \in [0,1] \mid x \in E_s \} \subset [0, s)$ (since it is connected and $x \in \bar{E}_s \subset E_0$) and so $\sigma(x) \leq s$. Therefore $x \in \overline{E}_s \setminus E_s$ implies that $\sigma(x) = s$.

\item If $\sigma(x) < s$ then there exists $t$ such that $\sigma(x) < t < s$. Then $x \notin E_t \supset \overline{E}_s$.

If $x \notin \overline{E}_s = \bigcap_{t < s} E_t$ then there exists $t < s$ such that $x \notin E_t$ and so $\sigma(x) < s$ since $\{ s \in [0,1] \mid x \in E_s \}$ is a connected interval and $x \in E_0$ by assumption.

\item If $\sigma(x) > s$ then there exists $t$ such that $x \in E_t$ and $\sigma(x) > t > s$. Therefore $x \in E_t \subset E_s$.

If $x \in E_s$ then $\sigma(x) \geq s$. The proof above shows that $\sigma(x) = s$ if and only if $x \in \overline{E}_s \setminus E_s$ and so $x \in E_s$ implies that $\sigma(x) > s$. \qedhere

\end{enumerate}
\end{proof}

\begin{corollary}\label{cor:sigma-continuous}
$\sigma$ is continuous.
\end{corollary}

\begin{proof}
The previous lemma shows that if $0 < s_1 < s_2 < 1$ then $\sigma^{-1}\left((s_1, s_2)\right) = E_{s_1} \setminus \overline{E}_{s_2}$ which is open. We also have $\sigma^{-1}([0, s_2)) = f^{-1}(c-\varepsilon) \setminus \overline{E}_{s_2}$ which is open and $\sigma^{-1}((s_1, 1]) = E_{s_1}$ which is also open. Therefore $\sigma^{-1}(U)$ is open for all open sets $U \subset [0,1]$ and so $\sigma$ is continuous.
\end{proof}

Now extend the domain of $\sigma$ to $f^{-1}([c-\varepsilon, c+\varepsilon]) \setminus W_C^+$ by defining $\sigma(x) = \sigma(\phi(x, \tau_{c-\varepsilon}(x)))$ so that $\sigma$ is constant on flow lines. Since $\tau_{c-\varepsilon}$ and $\phi$ are continuous (Lemmas \ref{lem:tau-continuous} and \ref{lem:sigma-continuous}) then $\sigma : f^{-1}([c-\varepsilon, c+\varepsilon]) \setminus W_C^+ \rightarrow [0,1]$ is also continuous.

\begin{lemma}\label{lem:continuous-above-critical-level}
For each $c \in C$ and any $\varepsilon > 0$ there exists a neighbourhood $V$ of $c$ such that $y \in V \setminus W_C^+$ implies that $\sigma(y) \in (1-\varepsilon, 1]$. 
\end{lemma}

\begin{proof}
Given any $\varepsilon > 0$, choose $s \in (1-\varepsilon, 1]$. Choose a neighbourhood $U$ of $W_x^- \cap f^{-1}(c-\varepsilon)$ such that $U \subset E_s$. Then Condition \eqref{cond:hyperbolicity} says that there exists a neighbourhood $V$ of $x$ such that for each $y \in V \setminus W_C^+$ we have $\phi(y, T) \in U \subset E_s$ for some $T > 0$. Therefore $\sigma(y) > s > 1 - \varepsilon$ by Lemma \ref{lem:sigma-continuous}. 
\end{proof}

We can further extend $\sigma$ to a function 
\begin{equation}
\sigma : \left( f^{-1}([c-\varepsilon, c+\varepsilon]) \setminus W_C^+ \right) \cup C \rightarrow [0,1] .
\end{equation}
by defining $\sigma(x) = \sigma(\phi(x, \tau_{c-\varepsilon}(x)))$ for $x \notin C$ and $\sigma(x) = 1$ for $x \in C$. Lemma \ref{lem:continuous-above-critical-level} shows that this is continuous. Define $g : f^{-1}([c-\varepsilon, c]) \rightarrow [c-3\varepsilon, c]$ by 
\begin{equation*}
g(x) = f(x) - 2 \varepsilon \sigma(x)
\end{equation*}
and define the set $Y := g^{-1}([c-3 \varepsilon, c-\varepsilon]) \subset f^{-1}([c-\varepsilon, c])$. Note that $g$ is continuous since $\sigma$ is continuous. Also note that $f$ and $g$ have the same stationary points for the flow $\phi(x,t)$, since $\frac{\partial}{\partial t} f(\phi(x, t)) = 0$ if and only if $x \in C$ and (by definition) $\sigma$ is constant along the flow, so $\frac{\partial}{\partial t} g(\phi(x, t)) = 0$ if and only if $x \in C$. In particular, there are no stationary points in $g^{-1}([c-\varepsilon, c])$ since $x \in C$ implies $g(x) = c - 2\varepsilon$.

\begin{figure}[ht]
\begin{center}
\begin{pspicture}(0,0.5)(10,3.5)
\psline[fillstyle=solid,fillcolor=lightgray](2,1)(3.5,3)(6.5,3)(8,1)(2,1)
\psline(0,1)(10,1)
\psline(3.5,3)(6.5,3)
\psline[linewidth=2pt](5,1)(5,3)
\psline(2,1)(3.5,3)
\psline(6.5,3)(8,1)
\psline[linestyle=dashed,linewidth=0.5pt](0,3)(10,3)
\psline[linestyle=dashed,linewidth=0.5pt](3.5,1)(3.5,3)
\psline[linestyle=dashed,linewidth=0.5pt](6.5,1)(6.5,3)
\psdots(5,3)(8,1)(6.5,1)(5,1)(3.5,1)(2,1)
\uput{4pt}[90](5,3){$C$}
\uput{4pt}[45](7.5,1.666){\small $g^{-1}(c-\varepsilon)$}
\uput{4pt}[135](2.5,1.666){\small $g^{-1}(c-\varepsilon)$}
\uput{4pt}[180](5,2){$W_C^-$}
\uput{4pt}[90](0,3){\small $f^{-1}(c)$}
\uput{4pt}[270](0,1){\small $f^{-1}(c-\varepsilon)$}
\uput{4pt}[270](8,1){\small $\sigma=0$}
\uput{4pt}[270](6.5,1){\small $\sigma=\frac{1}{2}$}
\uput{4pt}[270](5,1){\small $\sigma=1$}
\uput{4pt}[270](3.5,1){\small $\sigma=\frac{1}{2}$}
\uput{4pt}[270](2,1){\small $\sigma=0$}
\end{pspicture}
\end{center}
\caption{In the figure above, the flow lines for $f$ are represented by vertical lines and the level sets of $f$ are represented by horizontal lines. The set $Y$ is the shaded region.}
\end{figure}
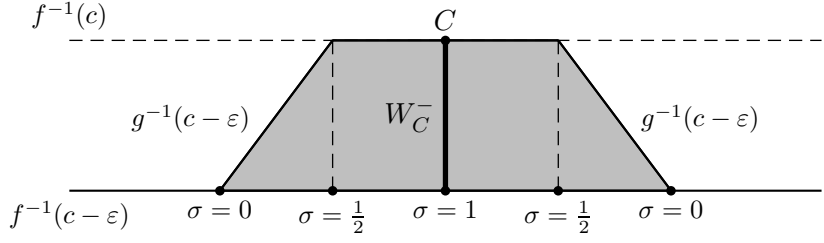

\begin{proposition}\label{prop:retract-to-Y}
Suppose that $f : Z \rightarrow \mathbb{R}$ satisfies conditions \eqref{cond:isolated-crit-values}, \eqref{cond:flow-alternative} and \eqref{cond:def-retract}. Then $Z_c \simeq Z_{c-\varepsilon} \cup Y$.
\end{proposition}

\begin{proof}
Since $Y = g^{-1}([c-3 \varepsilon, c-\varepsilon]) \subset f^{-1}([c-\varepsilon, c])$ then the proof reduces to defining a deformation retract of $f^{-1}([c-\varepsilon, c]) = g^{-1}([c-3\varepsilon, c])$ onto $g^{-1}([c-3\varepsilon, c-\varepsilon])$. Since $\frac{\partial}{\partial t} g(\phi(x, t)) = \frac{\partial}{\partial t} f(\phi(x, t)) < 0$ if and only if $x \notin C$ and there are no stationary points for the flow in $g^{-1}([c-\varepsilon, c])$, then 
\begin{itemize}

\item for all $x \in g^{-1}([c-\varepsilon, c])$ there exists $t$ such that $g(\phi(x, t)) < c - \varepsilon$, and 
 
\item for all $x \in g^{-1}(c-\varepsilon)$ we have $g(\phi(x,t)) < c-\varepsilon$ for all $t > 0$.

\end{itemize}
Therefore $g^{-1}(c-\varepsilon)$ is an \emph{exit set} for the flow on $g^{-1}([c-\varepsilon, c])$ (see \cite[Def. 2.2]{Conley78}). 

The function $g$ is continuous, so $g^{-1}(c-\varepsilon)$ is closed in $g^{-1}([c-\varepsilon, c])$, and we have already observed that there are no stationary points for the flow in $g^{-1}([c-\varepsilon, c])$. Therefore the required deformation retract follows from Wazewski's theorem (see for example \cite{Wazewski54} or \cite[Thm 2.3]{Conley78}).
\end{proof}

\section{Proof of the main theorem}\label{sec:main-theorem-proof}

Propositions \ref{prop:retract-to-critical-level} and \ref{prop:retract-to-Y} together define a deformation retract from $Z_b$ to $Z_c \cup Y$. In this section we complete the proof of Theorem \ref{thm:main-theorem-Morse-theory} by constructing a deformation retract from $Z_c \cup Y$ to $Z_{c-\varepsilon} \cup W_C^-$.

The basic idea is to construct this deformation retract by combining the gradient flow $\phi(x,t)$ and the deformation retract $r$ from condition \eqref{cond:def-retract}. Since these both preserve the space $Z$ then the composition of these deformation retracts will also preserve $Z$. The proof that the deformation retract is continuous uses condition \eqref{cond:hyperbolicity}. 

Given $x \in Y$, define $s_{final} : Y \rightarrow [0,1]$ by
\begin{equation*}
s_{final}(x) = \begin{cases} \frac{f(x)-(c-\varepsilon)}{2\varepsilon(1-\sigma(x))} & c-\varepsilon \leq f(x) < c-\varepsilon + 2\varepsilon(1-\sigma(x)) \\ 1 & f(x) \geq c-\varepsilon + 2 \varepsilon (1-\sigma(x))  \end{cases}
\end{equation*}
Note that if $f(x) < c-\varepsilon + 2\varepsilon(1-\sigma(x))$ then we automatically have $s_{final}(x) = \frac{f(x)-(c-\varepsilon)}{2\varepsilon(1-\sigma(x))} < 1$. If $\sigma(x) = 1$ then we have $f(x) \geq c-\varepsilon = c-\varepsilon + 2\varepsilon(1-\sigma(x))$ and so $s_{final}(x) = 1$. Define $f_{final} : Y \rightarrow [c-\varepsilon, c]$ by
\begin{equation*}
f_{final}(x) = \begin{cases} c-\varepsilon & c-\varepsilon \leq f(x) < c-\varepsilon + 2\varepsilon(1-\sigma(x)) \\ f(x) - 2\varepsilon(1-\sigma(x)) & f(x) \geq c-\varepsilon + 2 \varepsilon (1-\sigma(x)) \end{cases}
\end{equation*}
Note that it follows immediately from the definition that $c - \varepsilon \leq f_{final}(x) \leq f(x)$ for all $x \in Y$.

The functions $s_{final}$ and $f_{final}$ are chosen so that $s_{final}$ is the maximum value of $s$ to use in the definition of $y(x,s)$ in \eqref{eqn:definition-y} below, and $f_{final}$ is the value of $f(R(x,1))$ in the deformation retract $R(x,s)$ of \eqref{eqn:definition-R}. Note that $s_{final}(x) < 1$ implies that $f_{final}(x) = c-\varepsilon$.

Since $f_{final}$ is composed of continuous functions that agree when $f(x) = c-\varepsilon + 2 \varepsilon(1-\sigma(x))$ and the same is true for $s_{final}$ on the set $Y \setminus (W_C^- \cap f^{-1}(c-\varepsilon))$, then we have the following lemma.
 
\begin{lemma}
$s_{final}$ is continuous on $Y \setminus (W_C^- \cap f^{-1}(c-\varepsilon))$ and $f_{final}$ is continuous on $Y$.
\end{lemma}

\begin{definition}
For $s \in [0,1]$ and $x \in Y$, define 
\begin{equation*}
f_s(x) = \begin{cases} f(x) & s = 0 \\ \frac{s}{s_{final}(x)} f_{final}(x) + \left( 1 - \frac{s}{s_{final}(x)} \right) f(x) & 0 < s \leq s_{final}(x) \\ f_{final}(x) & s > s_{final}(x) \end{cases}
\end{equation*}
\end{definition}

\begin{remark}
Since $c - \varepsilon \leq f_{final}(x) \leq f(x)$ for all $x \in Y$ then $c - \varepsilon \leq f_s(x) \leq f(x)$ for all $x \in Y$ and $s \in [0,1]$.
\end{remark}

\begin{lemma}
The function $f_s(x)$ is continuous, it satisfies $f_{s_{final}(x)}(x) = f_{final}(x)$ and $f_s(x) \leq f(x)$ for all $(x,s) \in Y \times [0,1]$. 
\end{lemma}

\begin{proof}
The statement that $f_{s_{final}(x)}(x) = f_{final}(x)$ and $f_s(x) \leq f(x)$ for all $(x,s) \in Y \times [0,1]$ follows directly from the definition of $f_s$, and so it only remains to prove continuity. Since $s_{final}$ is continuous on  $Y \setminus (W_C^- \cap f^{-1}(c-\varepsilon))$ then the problem reduces to proving that $f_s$ is continuous at $W_C^- \cap f^{-1}(c - \varepsilon)$.

Given any $x \in W_C^- \cap f^{-1}(c- \varepsilon)$, note that $\sigma(x) = 1$ and so $f(x) \geq c - \varepsilon + 2 \varepsilon(1-\sigma(x))$, which implies that $s_{final}(x) = 1$ and $f_{final}(x) = c - \varepsilon = f(x)$. Therefore $f_s(x) = c - \varepsilon$ for all $s \in [0,1]$. 

Given $\delta > 0$, choose an open neighbourhood $U$ of $x$ in $Y$ such that $y \in U$ implies that $|f(y) - f(x)| < \delta$. Then $f(x) = c - \varepsilon \leq f_s(y) \leq f(y)$, and so $|f_s(y) - f_s(x)| < \delta$ for all $s \in [0,1]$ and all $y \in U$. Therefore $f_s(x)$ is continuous for $x \in W_C^- \cap f^{-1}(c-\varepsilon)$ and all $s \in [0,1]$.
\end{proof}

For convenience, define $y : Y \setminus C \rightarrow f^{-1}(c-\varepsilon)$ by 
\begin{equation}\label{eqn:definition-y}
y(x,s) = r \left( \phi(x, \tau_{c-\varepsilon}(x)), \min\{ s, s_{final}(x) \} \right) .
\end{equation}
where $r$ is the deformation retract of Condition \eqref{cond:def-retract} and $\tau$ is the map from Lemma \ref{lem:tau-continuous}. Note that if $f(x) \geq c - \varepsilon + 2 \varepsilon(1 - \sigma(x))$ then $s_{final}(x) = 1$ and so $y(x, 1) \in W_C^- \cap f^{-1}(c-\varepsilon)$. Now define the deformation retract $R : Y \times [0,1] \rightarrow Y$ by
\begin{equation}\label{eqn:definition-R}
R(x,s) = \begin{cases} \phi \left( y(x,s), \tau_{f_s(x)}(y(x,s)) \right) & x \in Y \setminus C \\ x & x \in C \end{cases}
\end{equation}
\begin{figure}[ht]
\begin{center}
\begin{pspicture}(0,0.5)(10,3.5)
\psline(0,1)(10,1)
\psline(3.5,3)(6.5,3)
\psline[linewidth=2pt](5,1)(5,3)
\psline(2,1)(3.5,3)
\psline(6.5,3)(8,1)
\psline[linestyle=dashed,linewidth=0.5pt](3.5,3)(0,3)
\psline[linestyle=dashed,linewidth=0.5pt](6.7,2.033)(0,2.033)
\psline[linestyle=dashed,linewidth=0.5pt](6.3,1.6)(0,1.6)
\psline[linewidth=0.5pt](6.7,2.033)(6.3,1.6)
\psline[linewidth=0.5pt](6.7,2.033)(6.7,1)
\psline[linewidth=0.5pt](6.3,1.6)(6.3,1)
\psdots(6.7,2.033)(6.3,1.6)(6.7,1)(6.3,1)
\psdots(5,3)
\uput{3pt}[0](6.7,2.033){\tiny $x$}
\uput{4pt}[305](6.7,1){\tiny $\phi(x, \tau_{c-\varepsilon}(x))$}
\uput{4pt}[240](6.3,1){\tiny $y(x,s)$}
\uput{3pt}[225](6.3,1.6){\tiny $R(x,s)$}
\uput{4pt}[90](5,3){$C$}
\uput{4pt}[45](7,2.333){$Y$}
\uput{3pt}[180](0,3){\tiny $f^{-1}(c)$}
\uput{3pt}[180](0,2.033){\tiny $f^{-1}(f(x))$}
\uput{3pt}[180](0,1.6){\tiny $f^{-1}(f_s(x))$}
\uput{3pt}[180](0,1){\tiny $f^{-1}(c-\varepsilon)$}
\end{pspicture}
\end{center}
\caption{The effect of the deformation retract $R$ on $x \in Y$.}
\end{figure}
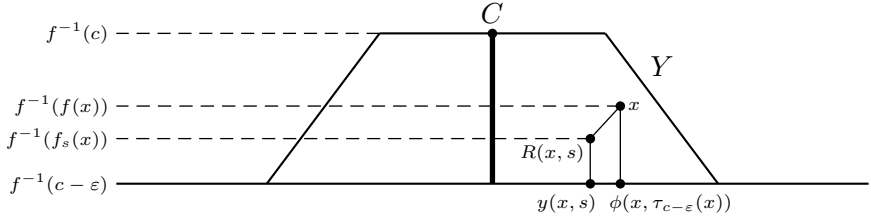

Note that $f(R(x,1)) = f_{final}(x)$ and so if $c-\varepsilon \leq f(x) < c-\varepsilon + 2\varepsilon(1-\sigma(x))$ then $R(x,1) \in f^{-1}(c-\varepsilon)$. If $f(x) \geq c-\varepsilon + 2 \varepsilon (1-\sigma(x))$ then $s_{final} = 1$ and so $R(x,1) \in W_C^-$.

\begin{figure}[ht]
\begin{center}
\begin{pspicture}(0,0.5)(10,3.5)
\psline(0,1)(10,1)
\psline(3.5,3)(6.5,3)
\psline[linewidth=2pt](5,1)(5,3)
\psline(2,1)(3.5,3)
\psline(6.5,3)(8,1)
\psline[linestyle=dashed,linewidth=0.5pt](4.5,3)(5,2.333)
\psline[linestyle=dashed,linewidth=0.5pt](4,3)(5,1.666)
\psline[linestyle=dashed,linewidth=0.5pt](3.5,3)(5,1)
\psline[linestyle=dashed,linewidth=0.5pt](3.25,2.666)(4.5,1)
\psline[linestyle=dashed,linewidth=0.5pt](3,2.333)(4,1)
\psline[linestyle=dashed,linewidth=0.5pt](2.75,2)(3.5,1)
\psline[linestyle=dashed,linewidth=0.5pt](2.5,1.666)(3,1)
\psline[linestyle=dashed,linewidth=0.5pt](2.25,1.333)(2.5,1)
\psline[linestyle=dashed,linewidth=0.5pt](5.5,3)(5,2.333)
\psline[linestyle=dashed,linewidth=0.5pt](6,3)(5,1.666)
\psline[linestyle=dashed,linewidth=0.5pt](6.5,3)(5,1)
\psline[linestyle=dashed,linewidth=0.5pt](6.75,2.666)(5.5,1)
\psline[linestyle=dashed,linewidth=0.5pt](7,2.333)(6,1)
\psline[linestyle=dashed,linewidth=0.5pt](7.25,2)(6.5,1)
\psline[linestyle=dashed,linewidth=0.5pt](7.5,1.666)(7,1)
\psline[linestyle=dashed,linewidth=0.5pt](7.75,1.333)(7.5,1)
\psdots(5,3)(8,1)(5,1)(2,1)
\uput{4pt}[90](5,3){$C$}
\uput{4pt}[45](7,2.333){$Y$}
\uput{4pt}[90](0,1){$f^{-1}(c-\varepsilon)$}
\uput{4pt}[270](8,1){\small $\sigma=0$}
\uput{4pt}[270](5,1){\small $\sigma=1$}
\uput{4pt}[270](2,1){\small $\sigma=0$}
\uput{3pt}[0](5,2.2){$W_C^-$}
\end{pspicture}
\end{center}
\caption{The deformation retract $R$ travels along the dashed lines to give a homotopy equivalence $Y \simeq f^{-1}(c-\varepsilon) \cup W_C^-$.}
\end{figure}
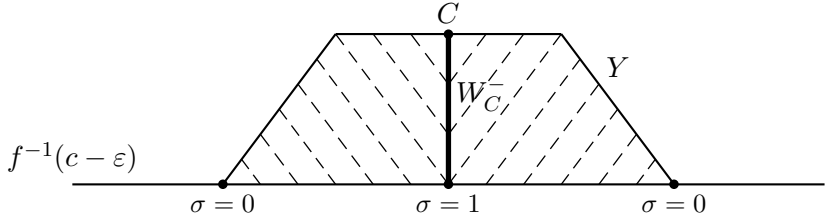

Now we can prove that $R$ is the desired deformation retract.

\begin{proposition}\label{prop:retract-to-cell-attachment}
Suppose that $f : Z \rightarrow \mathbb{R}$ satisfies conditions \eqref{cond:isolated-crit-values}--\eqref{cond:def-retract}. Then $Z_{c-\varepsilon} \cup Y \simeq Z_{c-\varepsilon} \cup W_C^-$.
\end{proposition}

\begin{proof}[Proof of Proposition \ref{prop:retract-to-cell-attachment}]

The proof reduces to showing that the deformation retract $R$ is continuous.

If $x \in Y \setminus \left( (W_C^- \cap f^{-1}(c-\varepsilon) ) \cup C \right)$ then this follows from the continuity of $s_{final}$, $f_{final}$ and $\tau$. Therefore the proof reduces to proving continuity at $x \in W_C^- \cap f^{-1}(c-\varepsilon)$ and $x \in C$.

\noindent \emph{Continuity at $x \in W_C^- \cap f^{-1}(c-\varepsilon)$.} 

Since $s_{final}$ is not continuous at $W_C^- \cap f^{-1}(c-\varepsilon)$ then we need to prove continuity by hand. Given any $x \in W_C^- \cap f^{-1}(c-\varepsilon)$ we need to show that for any neighbourhood $V$ of $x$ in $Y$ there exists a neighbourhood $V'$ of $x$ in $Y$ such that $(v,s) \in V' \times [0,1]$ implies that $R(v,s) \in V$. 

Since $V$ is open and the flow $\phi(\cdot,t)$ is continuous, then we can construct an open set in $V$ with a product structure as follows. Choose $\eta > 0$ and an open set $U \subset f^{-1}(c-\varepsilon)$ such that $y \in U$ implies that $\phi(y, \tau_\ell(y)) \in V$ for all $\ell \in [c-\varepsilon, c-\varepsilon + \eta)$. Define
\begin{equation*}
V_{\phi, \eta} := \left\{ z \in f^{-1}([c-\varepsilon, c-\varepsilon+\eta)) \, : \, \phi(z, \tau_{c-\varepsilon}(z)) \in U \right\} \subset V
\end{equation*}

Given such a $U \subset f^{-1}(c-\varepsilon)$, the continuity of the deformation retract $r$ from Condition \eqref{cond:def-retract} shows that for each $s \in [0,1]$ there exists $\delta_s > 0$ and a neighbourhood $U_s'$ of $x$ in $f^{-1}(c-\varepsilon)$ such that $r(v,t) \in U$ for all $(v,t) \in U_s' \times (s-\delta_s, s+\delta_s)$. Since $[0,1]$ is compact then there is a finite cover by open intervals of the form $(s_i - \delta_{s_i}, s_i + \delta_{s_i})$ for $i = 1, \ldots, n$. Therefore there exists a neighbourhood $U' = \cap_{i=1}^n U_{s_i}'$ of $x$ in $f^{-1}(c-\varepsilon)$ such that $r(v,t) \in U$ for all $v \in U'$ and $t \in [0,1]$. 

By continuity of the finite time flow, there exists an open subset $V' \subset f^{-1}([c-\varepsilon, c-\varepsilon + \eta))$ such that $v \in V'$ implies that $\phi(v, \tau_{c-\varepsilon}(v)) \in U'$ and hence $y(v, s) \in U$ for all $(v,s) \in V' \times [0,1]$. Note that $f_{s}(v) \leq f(v) < c-\varepsilon + \eta$ for all $v \in V'$ and $s \in [0,1]$ (using the same $\eta$ as above). Then $y(v,s) \in U$ for all $(v,s) \in V' \times [0,1]$ implies that $R(v,s) \subset V_{\phi, \eta} \subset V$ for all $(v,s) \in V' \times [0,1]$, completing the proof of continuity at $x \in W_C^- \cap f^{-1}(c-\varepsilon)$.

\noindent \emph{Continuity at $x \in C$.} 

Let $x \in C$ and let $V$ be any neighbourhood of $x$ in $Y$. For each $s \in [0,1]$, we want to show that there exists a neighbourhood $V' \times \left( (s-\delta, s+\delta) \cap [0,1] \right)$ of $(x,s)$ in $Y \times [0,1]$ such that $(x,s) \in V'$ implies that $R(x, s) \in V$. 

Given $V$, choose a neighbourhood $X$ of $x$ in $f^{-1}(c)$ and $\ell_0 \in (c-\varepsilon, c)$ such that 
\begin{equation*}
V_0 := \{ y \in Y \, : \, f(y) > \ell_0 \, \text{and} \, \varphi_{f(y), c}(y) \in X \} \subset V ,
\end{equation*}
where $\varphi_{f(y), c}$ is the map of level sets from Proposition \ref{prop:retract-to-critical-level}. In the proof of Proposition \ref{prop:retract-to-critical-level} we showed that $\varphi_{\ell, c} : f^{-1}(\ell) \rightarrow f^{-1}(c)$ is continuous, and so for $\ell = c-\varepsilon$ there exists a neighbourhood $U$ of $W_x^- \cap f^{-1}(c-\varepsilon)$ in $f^{-1}(c-\varepsilon)$ such that $\varphi_{\ell,c}(U) \subset V_0 \cap f^{-1}(c)$. The continuity of the deformation retract $r$ from Condition \eqref{cond:def-retract} together with the fact that $r$ is the identity on $W_x^- \cap f^{-1}(c-\varepsilon)$ shows that for each $s$ and any $\delta > 0$ there exists a neighbourhood $U'$ of $W_x^- \cap f^{-1}(c-\varepsilon)$ in $f^{-1}(c-\varepsilon)$ such that $r(y,t) \in U$ for all $y \in U'$ and $t \in (s-\delta, s+\delta)$. 

Now Condition \eqref{cond:hyperbolicity} shows that there exists a neighbourhood $V'$ of $x$ in $Y$ such that $\phi(y, \tau_{c-\varepsilon}(y)) \in U'$ for all $y \in V' \setminus C$. After shrinking $V'$ if necessary, we can assume that $s_{final} \equiv 1$ and $f_{final} > \ell_0$ on $V'$. By construction, $R$ maps $V' \times \left( (s-\delta, s+\delta) \cap [0,1] \right) \rightarrow U' \rightarrow U \rightarrow V_0 \subset V$, and so $R$ is continuous at $(x,s) \in C \times [0,1]$. 
\end{proof}

We are now ready to complete the proof of Theorem \ref{thm:main-theorem-Morse-theory}.

\begin{theorem}[Part (b) of Theorem \ref{thm:main-theorem-Morse-theory}]\label{thm:Morse-homotopy}
Suppose that $f : Z \rightarrow \mathbb{R}$ satisfies conditions \eqref{cond:isolated-crit-values}--\eqref{cond:def-retract} and suppose that $c \in (a, b)$ is the only critical value of $f$ in the interval $[a,b]$. Then $Z_b \simeq Z_a \cup W_C^-$.

Moreover, if a group $K$ acts on $Z$ such that $f$ is $K$-invariant and the deformation retract of Condition \eqref{cond:def-retract} is $K$-equivariant then the deformation retract from $Z_b$ to $Z_a \cup W_C^-$ is $K$-equivariant.
\end{theorem}

\begin{proof}
Combining the deformation retracts from Propositions \ref{prop:retract-to-critical-level}, \ref{prop:retract-to-Y} and \ref{prop:retract-to-cell-attachment} proves the first statement of the theorem.

If $f$ is $K$-invariant then the gradient flow is $K$-equivariant. If the deformation retract of Condition \eqref{cond:def-retract} is $K$-equivariant then the composition of deformation retracts from Propositions \ref{prop:retract-to-critical-level}, \ref{prop:retract-to-Y} and \ref{prop:retract-to-cell-attachment} is also $K$-equivariant, which proves the second statement of the theorem.
\end{proof}

\section{Conditions \eqref{cond:isolated-crit-values}--\eqref{cond:hyperbolicity} for the norm-square of a moment map on an affine variety}\label{sec:moment-map}

Let $G$ be a connected complex reductive algebraic group, and let $Z$ be an affine variety with an algebraic action of $G$. A result of Kempf \cite[Lemma 1.1]{Kempf78} shows that there is a representation $V$ of $G$ and a $G$-equivariant isomorphism from $Z$ to a closed affine subvariety of $V$ (which we also denote by $Z$), and so we can assume without loss of generality that $Z \subset V$ is an affine subvariety with a $G$-action induced from the linear action of $G$ on $V$. 

Let $K \subset G$ be the maximal compact subgroup, choose a $K$-invariant Hermitian inner product $\left< \cdot, \cdot \right>$ on $V$ and let $\omega(\cdot, \cdot) = \Im \left< \cdot, \cdot \right>$. Let $\rho_x : \mathfrak{g} \rightarrow T_x V \cong V$ denote the infinitesimal action of $G$ at any $x \in V$. Then the action of $K$ on $V$ is Hamiltonian with a moment map $\mu : V \rightarrow \mathfrak{k}^*$ given by
\begin{equation*}
\mu(x) \cdot u = \frac{1}{2} \omega(\rho_x(u), x) \quad \text{for all $u \in \mathfrak{k}$. }
\end{equation*}
It is easy to check that this satisfies the moment map equation $d \mu_x(X) \cdot u = \omega(\rho_x(u), X)$ (see for example \cite{Sjamaar98}). For any central element $\alpha \in Z(\mathfrak{k}^*)$ the function $\mu(x) - \alpha$ also satisfies the moment map equation; in the following we abuse the notation and absorb this constant into the function $\mu$. Define $f = \| \mu \|^2 : V \rightarrow \R$. Then the gradient flow of $f$ with initial condition $x \in V$ has the form $\phi(x,t) = g(t) \cdot x$ where $g(t) \in G$ satisfies $\frac{dg}{dt} g^{-1} = -i \mu(g(t) \cdot x)$ and $g(0) = \id$, and so $Z$ is preserved by the gradient flow, since it is preserved by $G$.

The goal of this section is to show that Conditions \eqref{cond:isolated-crit-values}--\eqref{cond:hyperbolicity} are satisfied for $f : Z \rightarrow \R$ in the setting described above. 

First we prove a general result showing that if Conditions \eqref{cond:isolated-crit-values}--\eqref{cond:hyperbolicity} are satisfied for $f : Z_1 \rightarrow \R$, then they are satisfied on any closed subset $Z_2 \subset Z_1$ preserved by the flow. In the case of moment maps on an affine variety described above, this lemma shows that it is sufficient to check Conditions \eqref{cond:isolated-crit-values}--\eqref{cond:hyperbolicity} for the ambient affine space $V$.

\begin{lemma}\label{lem:restrict-to-closed-subset}
Let $f : Z_1 \rightarrow \R$ be a function satisfying Conditions \eqref{cond:isolated-crit-values}--\eqref{cond:hyperbolicity} and let $Z_2 \subset Z_1 \subset M$ be any closed subset preserved by the gradient flow of $f$. Then $f : Z_2 \rightarrow \R$ satisfies Conditions \eqref{cond:isolated-crit-values}--\eqref{cond:hyperbolicity}. 
\end{lemma}

\begin{proof}
\begin{enumerate}

\item Since the critical values for $f : Z_1 \rightarrow \R$ are isolated, then the critical values for the restriction $f : Z_2 \rightarrow \R$ are also isolated.

\item Since $Z_2$ is preserved by the flow, then $x \in Z_2$ implies that $\phi(x,t) \in Z_2$ for all $t$. If $x \in Z_2$ and the limit $\lim_{t \rightarrow \infty} \phi(x,t)$ exists then it is contained in $Z_2$ since $Z_2 \subset Z_1$ is closed, and the same is true for $\lim_{t \rightarrow - \infty} \phi(x,t)$ if it exists.

\item The condition that the ambient manifold $M$ is analytic and $f : M \rightarrow \R$ is analytic is independent of $Z_1$ and $Z_2$.

\item Let $x \in Z_2$ be a critical point for $f : Z_2 \rightarrow \R$, and choose any $a < f(x)$ such that there are no critical values in $[a, f(x))$. Given any neighbourhood $U$ of $W_x^- \cap f^{-1}(a)$ in $f^{-1}(a) \cap Z_2$, let $X = (f^{-1}(a) \cap Z_2) \setminus U$. Then $X$ is closed in $Z_2$, so it is also closed as a subset of $Z_1$. Let $U' = (f^{-1}(a) \cap Z_1) \setminus X \subset Z_1$. Note that $U = U' \cap Z_2$. Then $U'$ is open in $f^{-1}(a)$ and so Condition \eqref{cond:hyperbolicity} on the space $Z_1$ shows that there exists an open neighbourhood $V'$ of $x$ in $Z_1$ such that $V' \setminus W_C^+$ flows into $U'$. Therefore, since the flow preserves $Z_2$, then $V := (V' \setminus W_C^+) \cap Z_2$ flows into $U = U' \cap Z_2$, and so Condition \eqref{cond:hyperbolicity} is satisfied for the restriction of $f : Z_2 \rightarrow \R$. \qedhere
\end{enumerate}
\end{proof}

Next we show that Conditions \eqref{cond:isolated-crit-values}--\eqref{cond:hyperbolicity} are satisfied for the norm-square of a moment map associated to a linear action. Combined with Lemma \ref{lem:restrict-to-closed-subset}, this shows that Conditions \eqref{cond:isolated-crit-values}--\eqref{cond:hyperbolicity} are satisfied for the norm-square of a moment map on any affine variety.

\begin{proposition}\label{prop:conditions1-4}
Let $G$ be a connected reductive Lie group acting linearly on $M = \C^n$ or $M = \C P^n$ and suppose that the action of the maximal compact subgroup $K$ is Hamiltonian with respect to the standard symplectic structure on $M$. Let $\mu : M \rightarrow \mathfrak{k}^*$ be a moment map for this action and define $f : M \rightarrow \R$ by $f(x) = \| \mu(x) \|^2$. Then Conditions \eqref{cond:isolated-crit-values}--\eqref{cond:hyperbolicity} are satisfied for $f : M \rightarrow \R$.
\end{proposition}

\begin{proof}

\begin{enumerate}

\item For projective varieties Condition \eqref{cond:isolated-crit-values} follows from Kirwan's explicit construction of the critical values in \cite{Kirwan84}. For affine varieties this follows from the analogous construction in \cite{Hoskins14}.

\item In the case $M = \C P^n$, the flow $\phi(x,t)$ exists for all $t \in \R$ since $M$ is compact. In particular, for any $a < b$ the subset $f^{-1}([a,b])$ is compact. Since $f$ is analytic (see below) then the Lojasiewicz inequality method of Simon \cite{Simon83} shows that the flow with initial condition $x \in f^{-1}([a,b])$ either converges to a critical value $c \in [a, b]$ or it flows out of the set $f^{-1}([a, b])$. This is true both forwards and backwards in time and so Condition \eqref{cond:flow-alternative} is satisfied.

For the case $M = \C^n$, the result of Sjamaar \cite[Lemma 4.10]{Sjamaar98} shows that the function $f = \| \mu \|^2$ is proper on each $G$-orbit. Since the gradient flow $\phi(x, t)$ is contained in $G \cdot x$, then for any $a < b$, the intersection of the gradient flow line with $f^{-1}([a,b])$ is compact. Therefore the same proof as the case $M = \C P^n$ shows that Condition \eqref{cond:flow-alternative} is also satisfied for $M = \C^n$.

\item The moment map associated to a linear representation is analytic, and so Condition \eqref{cond:analytic} is satisfied.

\item The function $f = \| \mu \|^2$ is minimally degenerate (see \cite{Kirwan84}). Therefore, for each critical point $x$, the unstable manifold $W_x^-$ has the structure of a graph over the negative eigenspace of the Hessian (see for example \cite{HirschPughShub77}). Around each critical point $x$, there is a neighbourhood $U$ and coordinates $(x_1, \ldots, x_n)$ such that the minimising manifold is given by $x_1 = \cdots = x_d = 0$, where $d$ is the number of non-negative eigenvalues of the Hessian of $f$ at $x$. Let $y = (x_1, \ldots, x_d)$ and $z = (x_{d+1}, \ldots, x_n)$. Following \cite[Ch. 10]{Kirwan84}, the negative gradient flow is given locally by
\begin{equation}\label{eqn:local-grad-flow}
\phi(x(0), t) = \left( \begin{matrix} y(t) \\ z(t) \end{matrix} \right) = \left( \begin{matrix} e^{Pt} y(0) + Y(t, y(0), z(0)) \\ e^{Qt} z(0) + Z(t, y(0), z(0)) \end{matrix} \right)
\end{equation}
where $P$ is diagonal with non-positive entries, $Q$ is diagonal with strictly positive entries, and $Y, Z$ and their first derivatives vanish at the origin. Clearly the linearised flow $y(t) = e^{Pt} y(0)$, $z(t) = e^{Qt} z(0)$ satisfies Condition \eqref{cond:hyperbolicity}, since the term $e^{Pt}$ does not increase the distance from the origin, while the term $z(t)$ strictly increases the distance from the origin. 

Then Condition \eqref{cond:hyperbolicity} for the nonlinear equation \eqref{eqn:local-grad-flow} follows from the fact that the flow is topologically conjugate to the linearised flow in a neighbourhood of the critical point $x$ (see \cite[Theorem 4.1]{HirschPughShub77}). \qedhere
\end{enumerate}

\end{proof}

In the remainder of the section, we show that Conditions \eqref{cond:isolated-crit-values}--\eqref{cond:hyperbolicity} impose extra conditions on the topology of the level sets of the function near a critical point.

Let $c$ be a critical value. Conditions \eqref{cond:isolated-crit-values} and \eqref{cond:flow-alternative} imply that if there are no critical values in $[\ell, c)$ then there is a well-defined function $\varphi_{\ell, c} : f^{-1}(\ell) \rightarrow f^{-1}(c)$ (defined in Proposition \ref{prop:retract-to-critical-level}) given by the gradient flow $\phi(x,t)$. Proposition \ref{prop:retract-to-critical-level} shows that if Condition \eqref{cond:analytic} is also satisfied then this map $\varphi_{\ell,c}$ is continuous. 

Define the equivalence relation $x_1 \sim x_2$ iff $\varphi_{\ell,c}(x_1) = \varphi_{\ell,c}(x_2)$, denote the quotient by $L_{\ell} = f^{-1}(\ell) / \negthickspace \sim$ with quotient map $q : f^{-1}(\ell) \rightarrow L_\ell$ and let $\xi_{\ell,c} : L_\ell \rightarrow f^{-1}(c)$ be the induced injective map
\begin{equation*}
\xymatrix{
f^{-1}(\ell) \ar[r]^{\varphi_{\ell,c}} \ar[d]^{q} & f^{-1}(c) \\
L_\ell \ar[ur]_{\xi_{\ell,c}} 
}
\end{equation*}
The next result shows that adding Condition \eqref{cond:hyperbolicity} imposes more structure on the maps $\varphi_{\ell,c}$ and $\xi_{\ell,c}$.

\begin{lemma}\label{lem:level-set-map-closed}
Let $f : Z \rightarrow \R$ be a function satisfying Conditions \eqref{cond:isolated-crit-values}--\eqref{cond:hyperbolicity}, let $c$ be a non-minimal critical value and choose any $\ell < c$ such that there are no critical values in $[\ell, c)$. Then the map $\varphi_{\ell,c}$ is closed and $\xi_{\ell,c}$ is a homeomorphism onto its image.
\end{lemma}

\begin{proof}
Let $A \subset f^{-1}(\ell)$ be closed. We aim to show that $\varphi_{\ell,c}(A)$ is closed. Let $x \in f^{-1}(c) \setminus \varphi_{\ell,c}(A)$. First consider the case where $x \in C:= \Crit_Z(f) \cap f^{-1}(c)$ is a critical point. Since $\varphi_{\ell,c}$ is continuous then $\varphi_{\ell,c}^{-1}(\{x\}) = W_x^- \cap f^{-1}(\ell)$ is closed in $f^{-1}(\ell)$, and by definition it is contained in $f^{-1}(\ell) \setminus A$, which is open. Therefore there is an open neighbourhood $U$ of $W_x^- \cap f^{-1}(\ell)$ contained in $f^{-1}(\ell) \setminus A$. Condition \eqref{cond:hyperbolicity} implies that there exists an open neighbourhood $V$ of $x$ in $f^{-1}(c)$ such that 
\begin{equation*}
V \subset \varphi_{\ell,c}(U) \subset \varphi_{\ell,c}(f^{-1}(\ell) \setminus A) \subset f^{-1}(c) \setminus \varphi_{\ell,c}(A)
\end{equation*}
Since the flow defines a continuous map $f^{-1}(c) \setminus C \rightarrow f^{-1}(\ell) \setminus W_C^-$, then if $x \notin C$ there exists an open neighbourhood $U$ of $x$ in $f^{-1}(c) \setminus C$ such that $\phi(y, \tau_\ell(y)) \in f^{-1}(\ell) \setminus A$ for all $y \in U$. Therefore $U \subset f^{-1}(c) \setminus \varphi_{\ell, c}(A)$.

Therefore $f^{-1}(c) \setminus \varphi_{\ell,c}(A)$ is open, and so $\varphi_{\ell,c}(A)$ is closed.

To see that $\xi_{\ell,c}$ is a homeomorphism onto its image, it is sufficient to show that it is injective, closed and continuous. It is injective by definition, and continuity and closedness follow from the associated properties of $\varphi_{\ell,c}$, since $L_{\ell}$ has the quotient topology.
\end{proof}

The following corollary of Lemma \ref{lem:level-set-map-closed} shows that Condition \eqref{cond:hyperbolicity} implies that the set of critical points with trivial unstable set (i.e. local minima) must be open in the level set.  
\begin{corollary}\label{cor:index-zero-open}
With the same conditions as Lemma \ref{lem:level-set-map-closed}, the set of critical points $x \in C$ such that $W_x^- = \{x\}$ is an open subset of $f^{-1}(c)$.
\end{corollary}

\begin{proof}
The level set $f^{-1}(\ell)$ is closed in $Z$. Then Lemma \ref{lem:level-set-map-closed} shows that $\varphi_{\ell,c}(f^{-1}(\ell))$ is closed in $f^{-1}(c)$. Since Condition \eqref{cond:flow-alternative} implies that any $x \in f^{-1}(c) \setminus C$ flows to $f^{-1}(\ell)$ then
\begin{equation*}
\{ x \in C \, : \, W_x^- = \{x\} \} = f^{-1}(c) \setminus \varphi_{\ell,c}(f^{-1}(\ell))
\end{equation*}
and so $\{ x \in C \, : \, W_x^- = \{x\} \}$ is open in $f^{-1}(c)$.
\end{proof}

Lemma \ref{lem:restrict-to-closed-subset} shows that Conditions \eqref{cond:isolated-crit-values}--\eqref{cond:hyperbolicity} are preserved on restricting to a closed subset preserved by the gradient flow. The following simple example shows that Condition \eqref{cond:hyperbolicity} fails if we do not assume that the subset is closed. In this example the main theorem of Morse theory also fails.

\begin{example}\label{ex:hyperbolicity-fails1}
On $\mathbb{R}_{\geq 0} \times [0, 2 \pi)$ define the equivalence relation $(0, \theta_1) \sim (0, \theta_2)$ and consider the quotient space $Z = \R_{\geq 0} \times [0, 2\pi) / \sim$. Note that the map $\varphi : Z \rightarrow \R^2$ given by $\varphi(\rho, \theta) = (\rho \cos \theta, \rho \sin \theta)$ is a bijection, but not a homeomorphism. 

Define the function $f : Z \rightarrow \R$ by $f(\rho, \theta) = \frac{1}{2} \left( \rho^2 \sin^2 \theta - \rho^2 \cos^2 \theta \right)$. Note that $f \circ \varphi^{-1}(x, y) = \frac{1}{2} ( y^2 - x^2 )$. Although $Z$ is not a manifold, we can still define a flow along which $f$ is strictly decreasing by $\phi(\rho_0, \theta_0, t) = \varphi^{-1}(e^{t} \rho_0 \cos \theta_0, e^{-t} \rho_0 \sin \theta_0)$. Note that the subset $\varphi^{-1} \{ (x, y) \in \R^2 \, : \, x > 0, y < 0 \}$ is preserved by this flow.

The flow has a single critical point $\{ \rho = 0 \}$. The unstable set is $W_0^- = \{ \theta = 0 \} \cup \{ \theta = \pi \}$. Given $\varepsilon > 0$, we have $f^{-1}(-\varepsilon) \cap W_0^- = \{ (\sqrt{\varepsilon}, 0) , (\sqrt{\varepsilon}, \pi) \}$.  Then any neighbourhood of $f^{-1}(-\varepsilon) \cap W_0^-$ in $f^{-1}(-\varepsilon)$ will not intersect the subset $\varphi^{-1} \{ (x, y) \in \R^2 \, : \, x > 0, y < 0 \} = \{ (\rho, \theta) \in Z \, : \, \rho > 0, \theta > \frac{3\pi}{2} \}$, however any neighbourhood of $\{ \rho = 0 \}$ will intersect this set, which is preserved by the flow. Therefore Condition \eqref{cond:hyperbolicity} must fail for this example.

One can also see that the main theorem of Morse theory must also fail for this example, since $f^{-1}((-\infty, -\varepsilon]) \cup W_0^-$ is disconnected, however $f^{-1}((-\infty, \varepsilon])$ is connected.
\end{example}

This example illustrates that Condition \eqref{cond:hyperbolicity} is a nontrivial assumption and that if we don't assume this condition then we need a replacement in order to ensure that the main theorem of Morse theory holds. 

\section{A compactness theorem for spaces of flow lines}\label{sec:compactness-flow-lines}

In this section we show that Conditions \eqref{cond:isolated-crit-values}--\eqref{cond:hyperbolicity} together with an extra compactness condition on the unstable sets imply that spaces of flow lines connecting two critical points can be compactified by spaces of broken flow lines. The same is true for spaces of flow lines connecting two critical sets if we also assume that the critical sets are compact. This type of theorem has appeared for Morse functions and Morse-Bott functions on smooth manifolds (see for example \cite{AustinBraam95}), however the techniques used in these papers rely on the manifold structure of the ambient space and the Morse-Bott assumption to explicitly describe the trajectories near the critical points. Here we give a completely different proof which is intrinsic to the singular space and which replaces the Morse-Bott assumption with the more general Condition \eqref{cond:hyperbolicity}. 

A related question is to construct a collar neighbourhood of the boundary of the space of flow lines and describe this explicitly in terms of spaces of broken flow lines. For Morse-Bott functions on smooth manifolds, this has been done in \cite{AustinBraam95}, and again the methods use the smooth structure of the ambient space, most notably in the assumption that the stable and unstable manifolds intersect transversally. In the paper \cite{wilkin-quiver-broken-flow-lines} we will combine the methods of this section with the results of \cite{wilkin-quiver-flow-hecke} to give an algebro-geometric description of this collar neighbourhood for the space of representations of a quiver with relations, again using a method intrinsic to the singular space. 

We assume throughout this section that Conditions \eqref{cond:isolated-crit-values}--\eqref{cond:hyperbolicity} hold. Since we always assume that $Z$ is a closed subset of a Riemannian manifold $M$, then for any $x \in Z$ there exists a neighbourhood $U$ such that for any $y_1, y_2 \in U$ the distance along the shortest geodesic from $y_1$ to $y_2$ is well-defined. In the remainder of this section we denote this distance by $\| y_1 - y_2 \|$, and we also use $\| \cdot \|$ to denote the length of a tangent vector in $T_y M$ for any $y \in Z$.

First we show that if the unstable set $W_x^-$ is locally compact, then Conditions \eqref{cond:isolated-crit-values}--\eqref{cond:analytic} imply that the intersection of the unstable set with a level set of $f$  is compact if the level set is close enough to the level set containing the critical point $x$. The local compactness condition is satisfied whenever $Z$ is a subset of a finite-dimensional manifold, and it is also satisfied for examples in gauge theory such as the Yang-Mills-Higgs functional on a compact Riemann surface, where the unstable set is a subset of a finite-dimensional manifold.

\begin{lemma}\label{lem:unstable-compact}
Suppose that $Z$ is locally compact and that $f : Z \rightarrow \mathbb{R}$ satisfies Conditions \eqref{cond:isolated-crit-values}--\eqref{cond:analytic}. Let $x$ be a critical point of $f$. Then for every $\varepsilon > 0$ such that there are no critical values in $[f(x)-\varepsilon, f(x))$, the set $W_x^- \cap f^{-1}(f(x) - \varepsilon)$ is compact.
\end{lemma}

\begin{proof}
Since the flow defines a continuous map $f^{-1}(f(x) - \varepsilon) \rightarrow f^{-1}(f(x))$ by Proposition \ref{prop:retract-to-critical-level}, and $W_x^- \cap f^{-1}(f(x) - \varepsilon)$ is the preimage of $x$ under this map, then $W_x^- \cap f^{-1}(f(x) - \varepsilon)$ is closed. Therefore, since $Z$ is locally compact, then it is sufficient to show that $W_x^- \cap f^{-1}(f(x) - \varepsilon)$ is contained in a bounded neighbourhood of $x$.

Since $Z$ is a subset of a manifold $M$ on which the gradient flow of $f$ is well-defined, and $f$ is analytic by Condition \eqref{cond:analytic}, then we can apply the Lojasiewicz inequality of \cite{Simon83} to show that there exists $\delta > 0$ and constants $C > 0$, $\theta \in (0, \frac{1}{2})$ such that for any $y$ satisfying $\| x - y \| < \delta$, we have
\begin{equation*}
\| \grad f(y) \| \geq C |f(x) - f(y)|^{1 - \theta} .
\end{equation*}
A standard calculation (cf. \cite{Simon83}) shows that for any flow line $y_t$ we have
\begin{equation*}
C \theta \| \grad f(y_t) \| \leq \frac{\partial}{\partial t} | f(x) - f(y_t) |^\theta
\end{equation*}
if $\| x - y_t \| < \delta$. Given any $y_0 \in W_x^-$ there exists $\tau < 0$ such that $\| x - y_t \| < \delta$ for all $t \leq \tau$. Integrating the above inequality on the interval $(-\infty, \tau]$ gives us
\begin{equation}\label{eqn:flow-line-bound}
\int_{-\infty}^\tau \| \grad f(y_t) \| dt \leq \frac{1}{C \theta} |f(x) - f(y_\tau)|^\theta \leq \frac{1}{C \theta} |f(x) - f(y_0)|^\theta .
\end{equation}
Now choose $\varepsilon' > 0$ such that $\frac{1}{C \theta} (\varepsilon')^\theta < \frac{1}{2} \delta$. Given $y_0 \in W_x^- \cap f^{-1}(f(x) - \varepsilon')$, suppose (for contradiction) that $\| x - y_t\| \geq \delta$ for some $t < 0$. Let $\tau = \inf \{ t \in \mathbb{R}_{< 0} \, : \, \| x - y_t \| \geq \delta \}$. Since the flow is continuous then $\| x - y_\tau \| = \delta$. Then the above estimate applies on the interval $(-\infty, \tau]$ and so we obtain a contradiction
\begin{equation*}
\delta = \| x - y_\tau \| \leq \int_{-\infty}^\tau \| \grad f(y_t) \| dt \leq \frac{1}{C \theta} |f(x) - f(y_0)|^\theta < \frac{1}{2} \delta  .
\end{equation*}
Therefore the set $\{ t \in \mathbb{R}_{< 0} \, : \, \| x - y_t \| \geq \delta \}$ must be empty, which implies that $\| x - y_0 \| < \delta$. Therefore we have shown that $W_x^- \cap f^{-1}(f(x) - \varepsilon')$ is closed and contained in a ball of radius $\delta$ around $x$, which implies that it is compact since $Z$ is locally compact.

Given any $\varepsilon > 0$ such that there are no critical values in $[f(x)-\varepsilon, f(x))$, the flow defines a homeomorphism $W_x^- \cap f^{-1}(f(x) - \varepsilon) \rightarrow W_x^- \cap f^{-1}(f(x) - \varepsilon')$, and so $W_x^- \cap f^{-1}(f(x) - \varepsilon)$ is also compact.
\end{proof}

Given two critical points $x_u$ and $x_\ell$ with $f(x_u) > f(x_\ell)$, define $\tilde{\mathcal{F}}(x_u, x_\ell) := W_{x_u}^- \cap W_{x_\ell}^+$ as the space of all points in $Z$ that flow up to $x_u$ and down to $x_\ell$. The flow defines a natural $\mathbb{R}$-action on $\tilde{\mathcal{F}}(x_u, x_\ell)$.

\begin{definition}\label{def:flow-line-space}
The \emph{space of flow lines connecting $x_u$ and $x_\ell$} is 
\begin{equation*}
\mathcal{F}(x_u, x_\ell) := \tilde{\mathcal{F}}(x_u, x_\ell) / \mathbb{R} .
\end{equation*}
\end{definition}
Given any $z \in (f(x_\ell), f(x_u))$, the space of flow lines $\mathcal{F}(x_u, x_\ell)$ is homeomorphic to $\tilde{\mathcal{F}}(x_u, x_\ell) \cap f^{-1}(z)$. Each flow line $\gamma \in \mathcal{F}(x_u, x_\ell)$ defines a map $\gamma : \R \rightarrow W_{x_u}^- \cap W_{x_\ell}^+$. Given two critical points $x_u, x_\ell$ we fix $z \in (f(x_\ell), f(x_u))$ and adopt the convention that $f(\gamma(t=0)) = z$ for all $\gamma \in \mathcal{F}(x_u, x_\ell)$. 
\begin{definition}\label{def:convergent-flow-lines}
A sequence $\{ \gamma_n \} \subset \mathcal{F}(x_u, x_\ell)$ converges to a limit $\gamma_\infty \in \mathcal{F}(x_u, x_\ell)$ if and only if the sequence $\{ \gamma_n(0) \}$ converges to a limit $\gamma_\infty(0) \in \tilde{\mathcal{F}}(x_u, x_\ell) \cap f^{-1}(z)$.
\end{definition}

\begin{remark}
This is stronger than the condition that $\{ \gamma_n(0) \}$ converges to a limit in $f^{-1}(z)$, since we also require that the limit is in $\tilde{\mathcal{F}}(x_u, x_\ell)$. Theorem \ref{thm:compactification} below shows that if $\{ \gamma_n(0) \}$ converges to a limit in $f^{-1}(z)$ then we can find a subsequence $\{ \gamma_{n_j} \}$ and interpret the limiting trajectory as a broken flow line.

Since the finite-time flow depends continuously on the initial condition, then $\{ \gamma_n \}$ converges to $\gamma_\infty \in \mathcal{F}(x_u, x_\ell)$ if and only if $\{ \gamma_n(t) \}$ converges to $\gamma_\infty(t) \in \tilde{\mathcal{F}}(x_u, x_\ell)$ for every $t \in \mathbb{R}$.
\end{remark}

In preparation for the main result of the section, we prove that the critical value at the lower endpoint of a flow line depends lower semi-continuously on the flow line.

\begin{lemma}\label{lem:local-closedness}
Suppose that $f : Z \rightarrow \mathbb{R}$ satisfies Conditions \eqref{cond:isolated-crit-values}--\eqref{cond:hyperbolicity}. Let $x_u, x_\ell$ be critical points with $f(x_u) > f(x_\ell)$ and let $\{ y_n \}$ be a sequence of points in $\tilde{\mathcal{F}}(x_u, x_\ell)$ which converges to $y_\infty \in Z$. Then $\lim_{t \rightarrow \infty} f(\phi(y_\infty,t)) \geq f(x_\ell)$ with equality if and only if $\lim_{t \rightarrow \infty} \phi(y_\infty, t) = x_\ell$.  
\end{lemma}

\begin{proof}
Suppose that $\lim_{t \rightarrow \infty} f(\phi(y_\infty,t)) = z < f(x_\ell)$. By Condition \eqref{cond:isolated-crit-values} we can choose $\varepsilon > 0$ such that there are no critical values in $(f(x_\ell), f(x_\ell)+\varepsilon]$. For each $n$ there exists $t_n$ such that $f(\phi(y_n, t_n)) = f(x_\ell) + \varepsilon$. Since $\lim_{t \rightarrow \infty} f(\phi(y_\infty,t)) = z < f(x_\ell)$ by assumption, then there also exists $t_\infty$ such that $f(\phi(y_\infty,t_\infty)) = f(x_\ell) + \varepsilon$. Since $\{ y_n \}$ converges to $y_\infty$ and the finite-time flow depends continuously on the initial condition, then $\phi(y_n, t_n)$ converges to $\phi(y_\infty, t_\infty)$. Since $y_n \in \tilde{\mathcal{F}}(x_u, x_\ell)$ for all $n$ then $\phi(y_n, t_n) \in W_{x_\ell}^+ \cap f^{-1}(f(x_\ell) + \varepsilon)$ for all $n$. 

Proposition \ref{prop:retract-to-critical-level} shows that the flow defines a continuous deformation retract from $f^{-1}(f(x_\ell) + \varepsilon)$ to $f^{-1}(f(x_\ell))$, therefore $W_x^+ \cap f^{-1}(f(x_\ell) + \varepsilon)$ is closed in $f^{-1}(f(x_\ell) + \varepsilon)$, and therefore $\phi(y_n, t_n) \rightarrow \phi(y_\infty, t_\infty)$ implies that $\phi(y_\infty, t_\infty) \in W_{x_\ell}^+$, contradicting the assumption that $\lim_{t \rightarrow \infty} f(\phi(y_\infty,t)) = z < f(x_\ell)$. 

If we assume that $\lim_{t \rightarrow \infty} f(\phi(y_\infty,t)) = f(x_\ell)$ then the same proof as above shows that $y_\infty \in W_{x_\ell}^+$. 
\end{proof}

The next theorem is the main result of this section, which shows that $\mathcal{F}(x_u, x_\ell)$ has a natural compactification by spaces of broken flow lines.

\begin{theorem}\label{thm:compactification}
Suppose that Conditions \eqref{cond:isolated-crit-values}--\eqref{cond:hyperbolicity} hold and that for any critical point $x$ the unstable set $W_x^-$ is locally compact. Let $x_\ell, x_u$ be two critical points with $f(x_u) > f(x_\ell)$, and let $\{ \gamma_n \}$ be a sequence of flow lines in $\mathcal{F}(x_u, x_\ell)$. Then there exists a subsequence $\{ \gamma_{n_j} \}$, a finite set of critical points $\{ x_0 = x_u, x_1, \ldots, x_m, x_{m+1} = x_\ell \}$ with $f(x_u) > f(x_1) > \cdots > f(x_m) > f(x_\ell)$ and a finite subset $\{ r_0, \ldots, r_m \}$ such that each $r_k \in (f(x_{k+1}), f(x_k))$ and the following property holds. For each $k = 0, \ldots, m$, define $t_{n_j}^{(k)}$ to be the unique time such that $\gamma_{n_j}(t_{n_j}^{(k)}) = r_k$. Then for each $k = 0, \ldots, m$ the sequence $\gamma_{n_j}(t_{n_j}^{(k)})$ converges to a point $y_k \in \tilde{\mathcal{F}}(x_k, x_{k+1})$. 
\end{theorem}

\begin{proof}
Using Condition \eqref{cond:isolated-crit-values}, choose $\varepsilon > 0$ such that there are no critical values in the interval $[f(x_u) - \varepsilon, f(x_u))$. For each $n$, let $t_n^{(0)}$ be the unique time such that $f(\gamma_n(t_n^{(0)})) = f(x_u) - \varepsilon$. Since $W_{x_u}^- \cap f^{-1}(f(x_u) - \varepsilon)$ is compact by Lemma \ref{lem:unstable-compact}, then there exists a subsequence $\{ \gamma_{n_j} \}$ and $y_0 \in W_{x_u}^-$ such that $\gamma_{n_j}(t_{n_j}^{(0)})$ converges to $y_0$. If $y_0 \in \tilde{\mathcal{F}}(x_u, x_\ell)$ then we are done. If not, then Lemma \ref{lem:local-closedness} shows that there exists a critical point $x_1$ with $f(x_u) > f(x_1) > f(x_\ell)$ such that $y_0 \in \tilde{\mathcal{F}}(x_u, x_1)$. Let $C_1$ denote the set of critical points with critical value $f(x_1)$.

Now use Condition \eqref{cond:isolated-crit-values} again to choose $\varepsilon > 0$ such that $[f(x_1) - \varepsilon, f(x_1))$ contains no critical values, define $s_{n_j}^{(1)}$ and $t_{n_j}^{(1)}$ as the unique real numbers such that $f(\gamma_{n_j}(s_{n_j}^{(1)})) = f(x_1)$ and $f(\gamma_{n_j}(t_{n_j}^{(1)})) = f(x_1) - \varepsilon$. Note that $\gamma_{n_j}(s_{n_j}^{(1)})$ converges to $x_1$ by Proposition \ref{prop:retract-to-critical-level}. Given any neighbourhood $U_1$ of $W_{x_1}^- \cap f^{-1}(f(x_1) - \varepsilon)$, Condition \eqref{cond:hyperbolicity} guarantees the existence of a neighbourhood $V_1$ of $x_1$ such that $V_1 \setminus W_{C_1}^+$ flows into $U_1$. Since $\gamma_{n_j}(s_{n_j}^{(1)})$ converges to $x_1$ then there exists $N_1$ such that $n_j > N_1$ implies that $\gamma_{n_j}(s_{n_j}^{(1)}) \in V_1$, and so $\gamma_{n_j}(t_{n_j}^{(1)}) \in U_1$. 

This is true for any neighbourhood $U_1$ of $W_{x_1}^- \cap f^{-1}(f(x_1) - \varepsilon)$, and so since $W_{x_1}^- \cap f^{-1}(f(x_1) - \varepsilon)$ is compact by Lemma \ref{lem:unstable-compact}, then there exists a further subsequence $\gamma_{n_j'}$ such that $\gamma_{n_j'}(t_{n_j'}^{(1)})$ converges to a point $y_1 \in W_{x_1}^- \cap f^{-1}(f(x_1)-\varepsilon)$. If $y_1 \in \tilde{\mathcal{F}}(x_1, x_\ell)$ then we are done. If not, then Lemma \ref{lem:local-closedness} shows that $y_1 \in \tilde{\mathcal{F}}(x_1, x_2)$ for some critical point $x_2$ with $f(x_2) > f(x_\ell)$, and we can repeat the above process inductively to obtain a finite set of critical points satisfying the conditions of the theorem.
\end{proof}

\section{Condition \eqref{cond:def-retract} for moment maps on affine varieties}\label{sec:condition5-moment-maps}

In this section we complete the proof of Theorem \ref{thm:main-theorem-affine-variety} by showing that Condition \eqref{cond:def-retract} holds for the case when $f = \| \mu \|^2$ is the norm-square of a moment map on an affine $G$-variety. The goal is to reduce the problem of proving that Condition \eqref{cond:def-retract} is satisfied to a simpler criterion on the analyticity of the unstable set $W_C^-$. Proposition \ref{prop:condition5} shows that this criterion holds for the norm-square of a moment map on an affine variety.

Let $G$ be a connected, reductive algebraic group and let $Z$ be an affine algebraic $G$-variety. By \cite[Lemma 1.1]{Kempf78}, we can assume that $Z \subset V$ where $V$ is a finite-dimensional vector space and the action of $G$ on $V$ is linear. Using the same notation and setup as in Section \ref{sec:moment-map}, let $f = \| \mu \|^2 : V \rightarrow \R$ be the norm-square of the associated moment map.

Let $c$ be a critical value of $f$ and let $C$ be the associated critical set $\Crit_Z(f) \cap f^{-1}(c)$. The unstable set of $C$ is
\begin{equation*}
W_C^- := \left\{ y \in V \, : \, \lim_{t \rightarrow - \infty} \gamma(y, t) \in C \right\} .
\end{equation*}
The goal of this section is to show that $W_C^-$ has the structure of a real analytic set in a neighbourhood of $C$, and hence we can then apply Theorem 1.1 of \cite{pflaumwilkin} to $W_C^- \cap f^{-1}(c-\varepsilon)$ in order to prove that the deformation retract of Condition \eqref{cond:def-retract} exists.

First we prove some results about the function $f$ on the smooth space $V$, before restricting to the subset $Z$ in Proposition \ref{prop:condition5}. Using an $\Ad$-invariant inner product on $\mathfrak{k}$, we identify $\mathfrak{k} \cong \mathfrak{k}^*$ and also use $\mu(x) \in \mathfrak{k}$ to denote the element corresponding to $\mu(x) \in \mathfrak{k}^*$. The infinitesimal action of $u \in \mathfrak{g}$ at $x$ is denoted $\rho_x(u) = \left. \frac{d}{dt} \right|_{t=0} e^{tu} \cdot x \in T_x V$. Note that the tangent space to the $G$-orbit through $x$ is $T_x (G \cdot x) = \im \rho_x$. Since $V$ is an affine space then (considering $V$ as a vector space with zero element at $x$) we can identify $V \cong T_x V$, and so for each $x$ we can consider $\im \rho_x$ as a subspace of $V$.

\begin{definition}\label{def:neg-slice-bundle}
Let $c$ be a critical value of $f$ and let $C = \Crit(f) \cap f^{-1}(c)$. The \emph{negative slice bundle} is
\begin{equation}\label{eqn:neg-slice-bundle-def}
S_C^- = \{ (x, y) \in C \times V \, : \, \lim_{t \rightarrow \infty} e^{i \mu(x) t} \cdot y = 0 \, \,  \text{and} \, \, y \perp \im \rho_x \}
\end{equation}
which is equipped with a canonical projection map $S_C^- \rightarrow C$. 
\end{definition}

The proof of the following lemma is contained in \cite[Sec. 4]{Kirwan84}.
\begin{lemma}
At each critical point $x \in C$, the fibre $S_x^- = \{ y \in V \, : \, \lim_{t \rightarrow \infty} e^{i \mu(x) t} \cdot y = 0, \, \, y \perp \im \rho_x \}$ is isomorphic to the negative eigenspace of the Hessian of $f$. Moreover, $S_C^-$ is a vector bundle over $C$, with constant rank on each connected component of $C$.
\end{lemma}

Therefore, in a neighbourhood of the critical set, we can identify the negative slice bundle with the normal bundle to the minimising manifold of the critical set via the map $S_C^- \rightarrow V$ given by $(x, y) \mapsto x + y$. Now we prove that the normal bundle is real analytic in $V$, and hence so is $S_C^-$.

\begin{lemma}\label{lem:neg-slice-analytic}
There is a neighbourhood $U$ of $C$ in $S_C^-$ such that for each $y \in U$ there is an open neighbourhood $U'$ of $y$ in $V$ and a finite set of real analytic functions $h_1, \ldots, h_n : U' \rightarrow \R$ such that $U \cap U' = \bigcap_{j=1}^n h_j^{-1}(0)$.
\end{lemma}

\begin{proof}
When $f = \| \mu \|^2$ is the norm-square of a moment map on a complex vector space, Hoskins \cite{Hoskins14} gives an explicit description of the Morse strata of $f$ based on Kirwan's description of the strata for moment maps on projective varieties in \cite{Kirwan84}. In particular, the strata are submanifolds $S_{K \cdot \beta} = G Y_\beta^{min}$, where $\beta = \mu(x) \in \mathfrak{k}$ for a critical point $x$ (the choice of $\beta$ is unique up to the adjoint action of $K$) and $Y_\beta^{min}$ is an open subset of the analytic set $V_+^\beta := \{ v \in V \, : \, \lim_{t \rightarrow \infty} e^{-i \beta t} \cdot v \, \, \text{exists} \}$. Since the $G$-action is analytic then the strata $S_{K \cdot \beta}$ are also analytic submanifolds. Moreover, the critical set $C_{K \cdot \beta} = S_{K \cdot \beta} \cap f^{-1}(c)$ is a real analytic subset.

Given a critical point $x \in C_{K \cdot \beta}$, there is a neighbourhood $U$ of $x$ and a projection $\pi : U \rightarrow S_{K \cdot \beta}$ such that $U$ is identified with a neighbourhood of the zero section of the normal bundle of $S_{K \cdot \beta}$ at $x$. Since $S_{K \cdot \beta}$ is an analytic submanifold then the normal bundle is analytic and so we can choose the projection $\pi$ to be analytic. 

Since the critical set $C_{K \cdot \beta}$ is analytic, then the preimage of $\pi^{-1} (U \cap C_{K \cdot \beta})$ is analytic. Therefore the negative slice bundle is locally cut out by analytic functions.
\end{proof}

For a fixed finite time $T$, the time $T$ flow defines a diffeormorphism $\phi_T : V \rightarrow V$. Since the flow is defined by the equation $\frac{dx}{dt} = - \grad f(x)$, for which the right-hand side depends analytically on $x$, then as explained in \cite{Nicolaescu10} (see also \cite[Ch. 1]{CourantHilbert89}) it follows from the Cauchy-Kowalevski theorem that the time $T$ solution to the flow depends analytically on the initial condition. This is summarised in the following lemma.

\begin{lemma}\label{lem:finite-flow-analytic}
For each finite $T > 0$, the time $T$ flow $\phi_T$ defines an analytic diffeomorphism $\phi_T : V \rightarrow V$.
\end{lemma}

Before we prove that Condition \eqref{cond:def-retract} holds for the norm-square of a moment map on an affine variety, we first review some results of Hubbard \cite{Hubbard05} on the analyticity of the unstable manifold for analytic flows which will be used in the proof of Proposition \ref{prop:condition5}.

In \cite{Hubbard05}, Hubbard considers the general case of an analytic flow on a vector space and shows that the unstable manifold of a critical point is analytically isomorphic to the negative eigenspace of the Hessian. In this notation of this paper, given a critical point $x \in C$, Hubbard defines a map on each fibre $\xi_x : S_x^- \cong T_x W_x^- \rightarrow W_x^-$ and shows in \cite[Thm. 6]{Hubbard05} that this is an analytic diffeomorphism in a neighbourhood of $x$. If we use $y$ to denote the coordinate on $S_x^-$ then Hubbard also shows that the power series in $y$ converges absolutely uniformly on this neighbourhood (cf. \cite[Thm. 6]{Hubbard05}) and that the map $\xi_x$ depends analytically on the critical point $x$ (cf. \cite[Thm. 12]{Hubbard05}). 

In order to apply these results in the proof of Proposition \ref{prop:condition5} below, first we complexify $V \hookrightarrow V^\C$ so that the real variables $x$ defining $V$ become complex variables $z$ defining $V^\C$, and with respect to these new variables the time $t$ flow $\phi_t$ (which is real analytic in the variable $x$ by Lemma \ref{lem:finite-flow-analytic}) becomes a flow $\phi_t^\C$ which is complex analytic in the variable $z$. Via the inclusion $V \hookrightarrow V^\C$, fixed points of the flow $\phi_t$ map to fixed points of the complexified flow $\phi_t^\C$. Let $C^\C$ denote the subset of the fixed point set of $\phi_t^\C$ containing the image of a critical set $C \subset V \hookrightarrow V^\C$ and let $W_z^-$ denote the unstable manifold of a critical point $z$. Applying \cite[Thm. 6]{Hubbard05} shows that for each $z \in C^\C$, there is a complex analytic map $\xi_z : T_z W_z^- \rightarrow W_z^-$ in a neighbourhood of $z$, and \cite[Thm. 12]{Hubbard05} shows that $\xi_z$ depends complex analytically on $z$.

If we use $w$ to denote the coordinate on $T_z W_z^-$, then these results show that $\xi_z(w)$ is separately complex analytic in $z$ and $w$, and therefore complex analytic as a function of $(z,w)$ by Hartogs' theorem (see for example \cite[Thm 2.2.8]{Hormander90}). If we restrict to the real locus defined by $(x,y) = (\Re(z), \Re(w))$ then the restriction $\xi_x(y)$ is real analytic as a function of $(x,y)$. Therefore we have proved the following lemma.

\begin{lemma}\label{lem:unstable-analytic}
There is a neighbourhood $U_1$ of the zero section of $S_C^- \rightarrow C$ and a neighbourhood $U_2$ of $C$ in $V$ together with an analytic homeomorphism $\xi : U_1 \rightarrow U_2 \cap W_C^-$.
\end{lemma}

\begin{remark}
It is not true in general that a function $f(x,y)$ which is separately real analytic in $x$ and $y$ is then real analytic as a function of $(x,y)$ (for example, consider $f(x,y) = \frac{xy}{x^2+y^2}$ as explained in \cite[p27]{Hormander90}). Therefore we have to apply Hubbard's theorem to the complexification $\phi_t^\C$ of the flow $\phi_t$ to prove the stronger result that our map $\xi_x(y)$ is the restriction of a map which is separately complex analytic in each variable. Then we can apply Hartogs' theorem to show that the map is analytic in both variables. The key is that Hubbard's results apply to the complexification of the flow.
\end{remark}

Now we can prove that Condition \eqref{cond:def-retract} holds for the norm-square of a moment map on an affine variety.

\begin{proposition}\label{prop:condition5}
Let $G$ be a connected, reductive algebraic Lie group, let $V$ be a representation of $G$ and let $Z \subset V$ be an affine $G$-variety. Suppose also that the action of the maximal compact subgroup $K$ on $V$ is Hamiltonian with respect to the standard symplectic structure on $V$. Let $\mu : V \rightarrow \mathfrak{k}^*$ be a moment map for this action and define $f : V \rightarrow \R$ by $f(x) = \| \mu(x) \|^2$. Then Condition \eqref{cond:def-retract} is satisfied for the restriction $f : Z \rightarrow \R$, and the deformation retract can be chosen to be $K$-equivariant.
\end{proposition}

\begin{proof}
Lemma \ref{lem:neg-slice-analytic} above shows that the negative slice bundle is locally cut out by analytic functions. Lemma \ref{lem:unstable-analytic} shows that (in a neighbourhood of the critical set $C$) there is an analytic homeomorphism $S_C^- \rightarrow W_C^-$ and therefore the unstable set $W_C^-$ is also locally cut out by analytic functions.

\emph{A priori} this neighbourhood of the critical set may not intersect $f^{-1}(c-\varepsilon)$, however for any point $y \in W_C^- \cap f^{-1}(c-\varepsilon)$ there exists a finite $T < 0$ such that $\phi(y, T) \in U$. Lemma \ref{lem:finite-flow-analytic} shows that the finite-time flow defines an analytic diffeomorphism onto its image, and so there exists a neighbourhood $U'$ of $y$ and a collection of functions $h_1, \ldots, h_n$ such that $W_C^- \cap U' = \bigcap_{j=1}^n h_j^{-1}(0)$. Since $f^{-1}(c-\varepsilon)$ is an analytic set, then we can extend the results of the previous paragraph to show that $W_C^- \cap f^{-1}(c-\varepsilon)$ is locally cut out by analytic functions.

Since the spaces $W_C^-$, $f^{-1}(c-\varepsilon)$ and $Z$ are all $K$-invariant analytic subsets of a finite-dimensional affine space $V$, then we can apply Theorem 1.1 of \cite{pflaumwilkin} to $Y = W_C^- \cap f^{-1}(c-\varepsilon) \cap Z$ and $X = f^{-1}(c-\varepsilon) \cap Z$ to show that Condition \eqref{cond:def-retract} holds for the function $f : Z \rightarrow \R$ and that the deformation retract can be chosen to be $K$-equivariant.
\end{proof}

\begin{remark}
In the above proof we use Hubbard's results to prove the existence of an analytic homeomorphism $S_C^- \rightarrow W_C^-$. A different question is whether this map restricts to a homeomorphism $S_C^- \cap Z \rightarrow W_C^- \cap Z$. This does not follow from the methods of \cite{Hubbard05}, however in \cite{wilkin-quiver-flow-hecke} we use the distance-decreasing property of the moment map flow to give a different construction of a map $S_C^- \rightarrow W_C^-$ which does restrict to a homeomorphism $S_C^- \cap Z \rightarrow W_C^- \cap Z$ and which is also $K$-equivariant. In the next section we explain how this fits into a larger picture of using Morse theory to compute topological invariants of quiver varieties.
\end{remark}

\section{Conclusion}

Theorem \ref{thm:main-theorem-Morse-theory} shows that the homotopy type of $Z_a = \{ x \in Z \, : \, f(x) \leq a \}$ changes by attaching a copy of $W_C^-$ as $a$ crosses a critical value $c$.  In order to relate the topological invariants of $Z_a$ and $Z_b$ for $a < c < b$, it is then natural to investigate the topology of the pair $(W_C^-, W_C^- \setminus C)$. 

When $f$ is a Morse function on a manifold $M$ with an appropriate compactness condition (e.g. $f$ satisfies the Palais-Smale Condition C), then one is faced with the same problem, and the solution is to use the Morse Lemma to relate the unstable manifold at a critical point $x$ to the negative eigenspace of the Hessian at $x$. Therefore the pair $(W_x^-, W_x^- \setminus \{x\})$ is homeomorphic to $(\R^\lambda, \R^\lambda \setminus \{0\})$, where $\lambda$ is the dimension of the negative eigenspace of the Hessian at $x$, and so the homotopy type of $Z_a$ changes by attaching a cell of dimension $\lambda$ as $a$ crosses the critical value $f(x)$. On a singular space this procedure fails because the Morse Lemma is no longer available and the definition of the Hessian does not make sense in the absence of the ability to take derivatives.

In this section we explain how to use the results of \cite{wilkin-quiver-flow-hecke} to describe $W_C^-$ in terms of analytic data around the critical set $C$ for the space $Z$ of representations of a quiver with relations. On this singular space the negative eigenspace of the Hessian is replaced by the intersection of $Z$ with the negative slice bundle $S_C^-$ of Definition \ref{def:neg-slice-bundle}, which is defined in terms of the group action and is therefore intrinsic to the singular space. Since the negative slice bundle has an explicit description \eqref{eqn:neg-slice-bundle-def} in terms of analytic data around the critical set then we can develop a procedure to compute topological invariants of the pair $(S_C^-, S_C^- \setminus C)$, which will be carried out in \cite{wilkin-quiver-topology}.

First we recall some basic definitions (cf. \cite{King94}, \cite{Nakajima94}) to set the notation for the rest of the section.
\begin{definition}
A \emph{quiver} $Q$ is a directed graph, consisting of vertices $\mathcal{I}$, edges $\mathcal{E}$, and head/tail maps $h,t : \mathcal{E} \rightarrow \mathcal{I}$.  A \emph{complex representation of a quiver} consists of a collection of complex Hermitian vector spaces $\{ V_i \}_{i \in \mathcal{I}}$, and $\C$-linear homomorphisms $\{ x_a : V_{t(a)} \rightarrow V_{h(a)} \}_{a \in \mathcal{E}}$. The \emph{dimension vector} of a representation is the vector ${\bf v} := ( \dim_\C V_i )_{i \in \mathcal{I}} \in \mathbb{Z}_{\geq 0}^\mathcal{I}$. The vector space of all representations with fixed dimension vector is denoted 
\begin{equation*}
\Rep(Q, {\bf v}) := \bigoplus_{a \in \mathcal{E}} \Hom(V_{t(a)}, V_{h(a)}) .
\end{equation*}
\end{definition}

The groups
\begin{equation*}
K_{\bf v} := \prod_{i \in \mathcal{I}} \U(V_i) \subset G_{\bf v} := \prod_{i \in \mathcal{I}} \GL(V_i, \C)
\end{equation*}
both act on the space $\Rep(Q, {\bf v})$ via the induced action on each factor $\Hom(V_{t(a)}, V_{h(a)})$ 
\begin{equation}\label{eqn:group-action}
( g_i )_{i \in \mathcal{I}} \cdot (x_a)_{a \in \mathcal{E}} := \left( g_{h(a)} x_a g_{t(a)}^{-1} \right)_{a \in \mathcal{E}} . 
\end{equation}
With respect to the standard symplectic structure on $\Rep(Q, {\bf v})$ induced from the symplectic structure on each $V_i$, the action of $K_{\bf v}$ is Hamiltonian with moment map
\begin{align}\label{eqn:moment-map-def}
\begin{split}
\mu : \Rep(Q, {\bf v}) & \rightarrow \mathfrak{k}_{\bf v}^* \\
 (x_a)_{a \in \mathcal{E}} & \mapsto \frac{1}{2i} \sum_{a \in E} [x_a, x_a^*] 
\end{split}
\end{align}
Given any $\alpha \in Z(\mathfrak{k}^*)$, define $f = \| \mu - \alpha \|^2$. In \cite{wilkin-quiver-flow-hecke} we adapt the method of \cite{wilkin-YMH-flow-lines} for Higgs bundles (based on the ``scattering method'' of Hubbard \cite{Hubbard05}) and use the distance-decreasing property of the gradient flow of $f$ to prove the following result.

\begin{theorem}[{\cite[Cor. 4.24]{wilkin-quiver-flow-hecke}}]\label{thm:scattering-homeomorphism}
For any critical set $C$ of $f = \| \mu - \alpha \|^2$ and any closed subset $Z \subset \Rep(Q, {\bf v})$ such that $G_{\bf v} \cdot Z \subset Z$, there is a $K_{\bf v}$-equivariant homeomorphism of pairs $(W_C^- \cap Z, (W_C^- \setminus C) \cap Z) \cong (S_C^- \cap Z, (S_C^- \setminus C) \cap Z)$.
\end{theorem}

This does not follow from the usual Banach fixed point theorem technique for constructing the unstable manifold of a Morse function on a smooth space (see for example \cite[Sec. 6]{Jost05}), since these methods rely on choosing local coordinates to define a projection onto the negative eigenspace of the Hessian at a critical point. Unless one can make this projection map compatible with the singularities in the space $Z$, then it is not obvious that both sides of $S_C^- \cong W_C^-$ remain homeomorphic after intersecting with a singular subset $Z \subset V$.

Theorem \ref{thm:scattering-homeomorphism} shows that for a $G_{\bf v}$-variety $Z \subset \Rep(Q, {\bf v})$, the problem of computing $K_{\bf v}$-equivariant topological invariants of the pair $(W_C^- \cap Z, (W_C^- \setminus C) \cap Z)$ reduces to the same question for the pair $(S_C^- \cap Z, (S_C^- \setminus C) \cap Z)$, which we can describe explicitly in terms of the negative slice at each critical point. This is the analog of the homeomorphism $(W_x^-, W_x^- \setminus \{x\}) \cong (\R^\lambda, \R^\lambda \setminus \{0\})$ for a Morse function on a manifold. 

An important class of examples is when $Z$ is the subvariety of representations satisfying a finite set of relations on the quiver. In the papers \cite{gothenwilkin} and \cite{wilkin-quiver-topology} we will investigate the topology of the pair $(S_C^- \cap Z, (S_C^- \setminus C) \cap Z)$ to derive information about the cohomology groups and low-dimensional homotopy groups of moduli spaces of representations of quivers with relations.


\end{document}